\documentclass[10pt]{article}

\usepackage[utf8]{inputenc}
\usepackage[T1]{fontenc}

\usepackage{epsf}
\usepackage{amsmath}

\allowdisplaybreaks

\usepackage[showframe=false]{geometry}
\usepackage{changepage}

\usepackage{epsfig}
\usepackage{amssymb}

\usepackage{amsthm}
\usepackage{setspace}
\usepackage{cite}
\usepackage{mcite}

\usepackage{algorithmic}  
\usepackage{algorithm}

\usepackage{shadow}
\usepackage{fancybox}
\usepackage{fancyhdr}

\usepackage{color}
\usepackage[usenames,dvipsnames,svgnames,table]{xcolor}
\newcommand{\bl}[1]{\textcolor{blue}{#1}}

\definecolor{mypurple}{rgb}{.4,.0,.5}

\usepackage[hyphens]{url}

\usepackage[colorlinks=true,
            linkcolor=black,
            urlcolor=blue,
            citecolor=purple]{hyperref}

\usepackage{breakurl}

\def\s{{\bf s}}

\def\y{{\bf y}}

\def\x{{\bf x}}

\def\x{{\mathbf x}}

\def\s{{\bf s}}

\def\x{{\bf x}}
\def\y{{\bf y}}

\def\tr{\mbox{Tr}}

\def\tr{{\rm tr}\,}

\def\be{\begin{equation}}
\def\ee{\end{equation}}
\def\ba{\left[\begin{array}}
\def\ea{\end{array}\right]}

\def\s{{\bf s}}

\def\x{{\bf x}}
\def\y{{\bf y}}

\def\1{{\bf 1}}

\def\0{{\bf 0}}

\def\cG{{\mathcal G}}

\def\A{{\bf A}}

\def\mR{{\mathbb R}}

\def\mN{{\mathbb N}}
\def\mE{{\mathbb E}}
\def\mS{{\mathbb S}}

\def\lp{\left (}
\def\rp{\right )}

\def\s{{\bf s}}

\def\y{{\bf y}}

\def\x{{\bf x}}

\def\x{{\mathbf x}}

\def\s{{\bf s}}

\def\x{{\bf x}}
\def\y{{\bf y}}

\def\tr{\mbox{Tr}}

\def\tr{{\rm tr}\,}

\def\be{\begin{equation}}
\def\ee{\end{equation}}
\def\ba{\left[\begin{array}}
\def\ea{\end{array}\right]}

\def\s{{\bf s}}

\def\x{{\bf x}}
\def\y{{\bf y}}

\def\({\left (}
\def\){\right )}

\def\1{{\bf 1}}

\def\0{{\bf 0}}

\def\cX{{\mathcal X}}

\def\cL{{\mathcal L}}

\def\cZ{{\mathcal Z}}
\def\cA{{\mathcal A}}

\def\cX{{\mathcal X}}

\usepackage{xcolor}
\usepackage{color}

\definecolor{darkgreen}{rgb}{0, 0.4,0}

\definecolor{purplebrown}{rgb}{0.5,0.1,0.6}

\definecolor{ultclupcol}{rgb}{0.1,0.5,0.5}

\definecolor{mytrycolor}{rgb}{0.5,0.7,0.2}

\definecolor{ultclupcola}{rgb}{.5,0,.5}

\definecolor{shadebrown}{rgb}{0.1,0.1,0.9}
\definecolor{lightblue}{rgb}{0.2,0,1}

\usepackage{fancybox}
\usepackage{graphicx}
\usepackage{epstopdf}
\usepackage{epsfig}
\usepackage{wrapfig}
\usepackage{subfigure}

\usepackage{xcolor}
\usepackage{tcolorbox}

\newtcbox{\xmybox}{on line,
arc=7pt,
before upper={\rule[-3pt]{0pt}{10pt}},boxrule=0pt,
boxsep=0pt,left=6pt,right=6pt,top=0pt,bottom=0pt,enhanced, coltext=blue, colback=white!10!yellow}

\newtcbox{\xmyboxa}{on line,
arc=7pt,
before upper={\rule[-3pt]{0pt}{10pt}},boxrule=0pt,
boxsep=0pt,left=6pt,right=6pt,top=0pt,bottom=0pt,enhanced, colback=white!10!yellow}

\newtcbox{\xmyboxb}{on line,
arc=7pt,
before upper={\rule[-3pt]{0pt}{10pt}},boxrule=1pt,colframe=darkgreen!100!blue,
boxsep=0pt,left=6pt,right=6pt,top=0pt,bottom=0pt,enhanced, colback=white!10!yellow}

\newtcbox{\xmyboxc}{on line,
arc=7pt,
before upper={\rule[-3pt]{0pt}{10pt}},boxrule=.7pt,colframe=blue!100!blue,
boxsep=0pt,left=6pt,right=6pt,top=0pt,bottom=0pt,enhanced, coltext=blue, colback=white!10!yellow}

\newtcbox{\xmytboxa}{on line,
arc=7pt,
before upper={\rule[-3pt]{0pt}{10pt}},boxrule=.0pt,colframe=pink!50!yellow,
boxsep=0pt,left=6pt,right=6pt,top=0pt,bottom=0pt,enhanced, coltext=white, colback=blue!40!red}

\newtcbox{\xmytboxb}{on line,
arc=7pt,
before upper={\rule[-3pt]{0pt}{10pt}},boxrule=.0pt,colframe=pink!50!yellow,
boxsep=0pt,left=6pt,right=6pt,top=0pt,bottom=0pt,enhanced, coltext=white, colback=white!40!green}

\makeatletter
\newcommand\subsubsubsection{\@startsection{paragraph}{4}{\z@}{-2.5ex\@plus -1ex \@minus -.25ex}{1.25ex \@plus .25ex}{\normalfont\normalsize\bfseries}}
\newcommand\subsubsubsubsection{\@startsection{subparagraph}{5}{\z@}{-2.5ex\@plus -1ex \@minus -.25ex}{1.25ex \@plus .25ex}{\normalfont\normalsize\bfseries}}
\makeatother

\newtheorem{theorem}{Theorem}

\newtheorem{corollary}{Corollary}

\newtheorem{remark}{Remark}

\begin{document}

\begin{singlespace}

\title {Proving Lehner's formulas via RDT -- indefinite and asymmetric scenarios
}
\author{
\textsc{Mihailo Stojnic
\footnote{e-mail: {\tt flatoyer@gmail.com}}
}}
\date{}
\maketitle

\centerline{{\bf Abstract}} \vspace*{0.1in}

The strong asymptotic freeness established in \cite{HaagThor05, Schultz05} allows for the study of norms of Gaussian matrix polynomials via corresponding free operator counterparts. For spectral edges of semicircular free counterparts to symmetric Kronecker-Gaussian matrices, Lehner in \cite{Lehner99} determined closed-form analytical characterizations. As an alternative to classical spectral methods, in  \cite{Stojnicleh26}, we created a Random Duality Theory (RDT) based framework for studying these problems and reproved Lehner's formula for definite matrix coefficients.

In this work, we develop the RDT machinery further and achieve strong progress in several key directions: (i) Symmetric Square Case: We consider indefinite matrix coefficients and provide spectral edges lower bounds that match Lehner's formula. (ii) Asymmetric Non-Square Variants: We establish RDT asymmetric analogues to Lehner formulas and prove that they lower-bound the spectral edges. (iii)  Interlaced Decoupling Property: We uncover that the obtained analogues exhibit a remarkable interlaced decoupling property. 

Additionally, we consider semi-definite programming (SDP) formulations that allow for the practical solution of the obtained analytical characterizations. Theoretical predictions obtained through solving SDPs are compared to numerical simulations, showing a striking agreement even for problem dimensions on the order of a few hundreds.

\vspace*{0.25in} \noindent {\bf Index Terms: Strong asymptotic freeness; Lehner's formula; Random duality theory (RDT)}.

\end{singlespace}

\section{Introduction}
\label{sec:back}

Remarkable breakthroughs \cite{Schultz05, HaagThor05} allowed the move from \emph{weak} \cite{Voic91} to \emph{strong asymptotic} freeness and showed that the spectra of Gaussian matrix polynomials converge towards their semicircular free counterparts. Further developments soon followed, including establishing analogous results for different random \cite{Anderson13} or deterministic plus random (Wigner/GUE/Haar/permutations) scenarios \cite{Parraud22, CollMale14, Male12, GuionnShly09, CollinsGP22, Belinschi17}.

The range of applications expanded rapidly as well. Studies of various geometric \cite{HideMagee23}, graph-related \cite{LubPer16, BorCollins19}, and algebraic phenomena \cite{Hayes22} became possible, often allowing for the resolution of many important, long-standing conjectures and open problems \cite{Hayes22, HaagThor05, LubPer16, BorCollins19}.

Particularly strong progress was achieved very recently in \cite{bandfree23, bandfree24}, where \emph{intrinsic} freeness was introduced as a finite-dimensional asymptotic freeness counterpart. It allows to circumvent asymptotic scenarios and analyze spectral features of a very wide class of Gaussian matrices. For a diverse range of applications related to spiked models, sample covariances, graph properties, and more, please see \cite{bandfree23, bandfree24, HandelSurvey26}.

\section{Preliminaries and contextualization}
\label{sec:sampcov}

Practical relevance and utilization of the freeness concepts ultimately relies on one's ability to analyze the associated free equivalent models. In that regard, a series of results established by Lehner \cite{Lehner99,Lehner98,Lehner01} can be viewed as trailblazing. In many practical scenarios, the characterizations established therein remain difficult to improve upon.

Of particular interest for our considerations are elegant deterministic spectral edges formulas for a class of semicircular free operators obtained in \cite{Lehner99} (by the strong asymptotic freeness, they fully characterize the spectral edges of the Kronecker–Gaussian ensembles). Namely, in \cite{Stojnicleh26}, we circumvented random matrix theory and developed a different methodology to determine Kronecker-Gaussian matrices spectral edges and ultimately reprove Lehner's formula from \cite{Lehner99}. Most notably, we showed that relying solely on concepts utilized within \emph{Random Duality Theory} (RDT) \cite{StojnicCSetam09, StojnicRegRndDlt10, Stojniccovmat26} is fully sufficient in that aspect.

However, the results of \cite{Stojnicleh26} treat the so-called definite variant of Lehner's formulas. Here, we develop the RDT machinery of \cite{Stojnicleh26} further so that (i) \emph{indefinite} and (ii) \emph{asymmetric} Lehner's formulas analogues can be treated as well.

\subsection{Mathematical setup}
\label{sec:context}

To properly explain and contextualize our main contributions, we first introduce a few mathematical preliminaries, starting with precise definitions of the key objects studied below.

We consider positive integers $m,n,k_1,k_2,l\in\mathbb{N}$ and operate in a high-dimensional regime where $n$ and $m$ are large, with $\alpha=\lim_{n\rightarrow \infty}\frac{m}{n}$. While the analysis allows for $k_1$, $k_2$, and $l$ to depend on $n$, we treat them as fixed to keep the presentation focused on conceptual insights. The primary source of randomness in our study is a collection of independent standard normal matrices $G_i=G_i^T\in\mathbb{R}^{m\times n}, i=1,\dots,l$. When $m=n$, we associate them with their symmetric counterparts $\bar{G}_i=\bar{G}_i^T\in\mathbb{R}^{n\times n}, i=1,\dots,l$.
\begin{equation}\label{eq:amat1a0}
 \bar{G}_i = \frac{1}{\sqrt{2}} \lp G_i + G_i^T\rp, i=1\dots,l. 
 \end{equation}
Let $A_i \in \mathbb{R}^{k_1 \times k_2},i=0,1,\dots,l$ be a collection of deterministic matrices. When $k_1=k_2=k$, we also associate with it its symmetric (a priori indefinite) counterpart $A_i=A_i^T \in \mathbb{R}^{k \times k},i=0,1,\dots,l$. We will denote the standard identity matrix by $I$. Its size will be clear from the context, though we may occasionally emphasize the size by adding a subscript.

Considering the symmetric scenario (where $k_1=k_2=k$ and $m=n$), the following sum of Kronecker-Gaussian matrices is of critical importance in studying the spectral norms of matrix polynomials \cite{HaagThor05,CollinsGP22,Schultz05}
\begin{equation}\label{eq:amat1a0b0}
 H = A_0\otimes I + \frac{1}{\sqrt{n}} \sum_{i=1}^{l} A_i\otimes \bar{G}_i. 
 \end{equation}
A main takeaway associated with the asymptotic freeness is that the spectrum of $H$  for large $n$ approaches the spectrum of the following semicircular free equivalent
\begin{equation}\label{eq:amat1a0b0}
 H_{free} = A_0\otimes I + \sum_{i=1}^{l} A_i\otimes \s_i, 
 \end{equation}
with $\s_i$ being from a semicircular family.

 Let   $\lambda_i(\cdot)$ denote the $i$-th smallest eignevalue of its argument. One can then write
 \begin{equation}\label{eq:amat1a0b1}
 \lambda_1(H) \leq  \lambda_2(H) \leq \dots  \leq  \lambda_n(H) , 
 \end{equation}
 with $ \lambda_1(H)$ and  $ \lambda_n(H)$ being the smallest and the largest eigenvalues of $H$, respectively. Consider an array of matrices 
 \begin{equation}
 \label{eq:aseteq1}
 \cA \triangleq \{A_i\}_{i=0,1,\dots,l}. 
 \end{equation}
The above mentioned  Lehner's  deterministic spectral edge characterization of $H_{free}$ is given by  \cite{Lehner99}
\begin{eqnarray}\label{eq:amat1a0b2}
 \rho_n = \rho_n(\cA)  &  \triangleq & \inf_{Z\succeq 0 } \lambda_n\lp A_0\otimes I + Z+   \sum_{i=1}^{l} A_iZ^{-1}A_i\rp 
 \nonumber \\
 \rho_1= \rho_1(\cA) & \triangleq & \sup_{0\succeq Z } \lambda_1\lp A_0\otimes I + Z+   \sum_{i=1}^{l} A_iZ^{-1}A_i\rp  . 
 \end{eqnarray}

\subsection{Relevant prior work}
\label{sec:context1}

\noindent \underline{\emph{Symmetric scenarios -- asymptotic freeness:}} By the strong asymptotic freeness \cite{HaagThor05,Schultz05}, (\ref{eq:amat1a0b2}) also characterizes the $H$'s spectral edges. In particular, keeping in mind the above-mentioned spectrum closeness between $H$ and $H_{free}$, one has
 \begin{eqnarray}\label{eq:amat1a0b3}
\lim_{n\rightarrow \infty} \mE \lambda_1(H) = \rho_1 \quad \quad \mbox{and} \quad \quad
\lim_{n\rightarrow \infty} \mE \lambda_n(H) = \rho_n , 
 \end{eqnarray}
 with very strong concentrations as well. 

In fact, various types of connections between $\rho_1$ and $\rho_n$ and $\lambda_1(H)$ and $\lambda_n(H)$ have been proven in \cite{ChenGVTH26,Anderson13,bandfree23,bandfree24,HaagThor05,Schultz05,CollinsGP22} through several different analytical methodologies. They include standard free probability Stieltjes transform/matrix Dyson equations \cite{HaagThor05,Schultz05,CollinsGP22} and first order \cite{BorCollins19,BordColl24,Anderson13} based methods as well as matrix concentrations and interpolations \cite{CollinsGP22,bandfree23,bandfree24,ChenGVTH26}. (For studies particularly tailored for permutations/unitary ensembles, see \cite{BorCollins19,Magee25,Cassidy24,BordColl24,CollMale14}; for more general (including non-Gaussians) ensembles, see, e.g., \cite{GuionnShly09,BrailUniv24}).

\vspace{.15in}
\noindent \underline{\emph{Symmetric scenarios -- intrinsic freeness:}} In the regimes of our interest, concentrations are induced by large $n$. However, that is not necessarily required. Particularly relevant in this regard is the recent breakthrough by \cite{bandfree23,bandfree24}, which introduces the concept of \emph{intrinsic} freeness.

Unlike \cite{HaagThor05,CollinsGP22,Schultz05}, which consider the $n\rightarrow \infty$ regime, \cite{bandfree23,bandfree24} assume $n=1$ and establish upper and lower bounds on the spectral free equivalents. In doing so, they recognize the critical role that $\|\mbox{cov}(H)\|$ and $k$ play in concentrations. While the choice of $n=1$ might seem like a deviation from the block structure, it allows for the consideration of various options for the covariance of $H$ that frequently appear in practical applications \cite{bandfree23,bandfree24}. This approach is also particularly useful for matrices with generic variance profiles \cite{BandHand16, VanHandelSpec17, LatHandYouss18, CaiHanZhang22,Husson22,Ducatez24,HussMcKenna24,BrailUniv24,LeeLee24}.

A wide range of other intrinsic freeness applications is discussed in \cite{bandfree23,bandfree24}. We would specifically single out those associated with the algorithmic and information-theoretic properties of spiked models and sample covariances. The results in \cite{bandfree23,bandfree24} allow for the precise determination of key features, such as residual estimation errors and BBP phase transitions. These precise results stand in stark contrast to usual scaling-order characterizations \cite{KolLou17,Tropp11,VershNonAsym12,KanLovSim97, Bourgain99, Rudelson99, GianHarTso05, Paouris06, Adam10,Adamczak11}. For further details on alternative precise analytical methodologies, please see, for example, \cite{BBP05,MontRich14,DonGavJohn18,Lesetal17,PerryWB20,BarbMM17,PourBM24,BehneReeves22,PakKK23,Stojniccovmat26}.

\vspace{.15in}
\noindent \underline{\emph{Asymmetric scenarios:}} A majority of the above-mentioned freeness work is related to symmetric scenarios — specifically the one described by (\ref{eq:amat1a0b0}) with $A_i=A_i^T$. More closely related to some portions of our work would be the analogous asymmetric ones.

While simple algebraic transformations allow translations from symmetric to asymmetric scenarios, to the best of our knowledge, direct spectral methods analytical treatments of asymmetric scenarios resulting in formulations as elegant/useful as (\ref{eq:amat1a0b2}) seem to be scarce in the existing literature. A notable exception is \cite{ParHand25}, where the extreme singular values of a set of free operators associated with intrinsic freeness are characterized. One would then expect that similar results apply to asymptotic freeness, but we are unaware of the spectral methods works that formally and directly establish it.

We should also add that \cite{bandfree23} demonstrated that intrinsic freeness can be combined with asymptotic freeness, so that the results of \cite{Schultz05, HaagThor05} extend to a wide range of Gaussian matrices beyond standard normals (this is, however, typically stated in the squared context).

\vspace{.15in}
\noindent \underline{\emph{Different analytical approaches:}}  Finally, we should mention approaches that differ from the methods discussed above. Utilizing the Sudakov-Fernique \cite{Sudakov71,Fernique74,Fernique75,Vitale00} comparison technique and assuming $A_i\succeq 0,i>0$, \cite{CollYama26} showed that negligibly modified $\rho_1$ and $\rho_n$ are the lower and upper bounds on $\lim_{n\rightarrow \infty} \mE \lambda_1(H) $ and $\lim_{n\rightarrow \infty} \mE \lambda_n(H) $. A nice additional observation of \cite{CollYama26} regarding the orthogonal invariance actually allows to slightly expand the definite set $A_i\succeq 0,i>0$ into a set that fits within the Kraus decomposition.

In \cite{Stojnicleh26}, we proved (\ref{eq:amat1a0b3}) relying on comparison concepts as well. However, we leaned on the methodologies utilized within Random Duality Theory (RDT) \cite{StojnicCSetam09,StojnicRegRndDlt10,Stojniccovmat26} and proved for $A_i\succeq 0,i>0$ that $\rho_n$ is not only the $\lim_{n\rightarrow \infty} \mE\lambda_n(H) $'s upper bound, but also its exact value (the same then automatically translates to $\rho_1$ and $\lim_{n\rightarrow \infty} \mE \lambda_1(H) $).

\subsection{Our contributions}
\label{sec:context2}

We have further developed the RDT machinery, which allows us to make strong progress in several directions. More concretely, we have:
\begin{itemize}
 \item Symmetric Square Case: We first focus on the symmetric square case ($k_1=k_2=k$, $m=n$, and $A_i=A_i^T$), which corresponds to Lehner's formula from (\ref{eq:amat1a0b2}). Differently from \cite{Stojnicleh26}, where $A_i\succeq 0,i>0$ was assumed, we consider indefinite matrix coefficients (i.e., we do not assume a priori that $A_i\succeq 0,i>0$) and provide spectral edges lower bounds that match Lehner's formula (see Section \ref{sec:symmin}).
 \item Asymmetric Non-Square Variants: We then focus on the asymmetric non-square variants of Kronecker-Gaussian matrices (where $k_1\neq k_2$, $m\neq n$, and $A_i\neq A_i^T$ choices are allowed). While simple algebraic tricks allow translation from symmetric to nonsymmetric scenarios, our focus is on the results that the RDT methodology can achieve. In that regard, we establish RDT versions of the asymmetric analogues to Lehner formulas and prove that they lower-bound the spectral edges (see Section \ref{sec:asm}). Several particular scenarios are considered: (i)  Centered: Assumes $A_0=0$ (see Section \ref{sec:handlerdA0}). (ii)   Centered Deformed: Assumes both $A_0=0$ and deformation $A_D\otimes I + H^TH$, resembling Wishart deformations and sample covariances (see Section \ref{sec:defnqhandlerdA0}). (iii)  Non-Centered: Assumes $A_0\neq 0$ and corresponds to the full asymmetric analogue to the classical centered symmetric ensembles from \cite{HaagThor05,Schultz05} (see Section \ref{sec:nqhandlerdA0}).
  \item Interlaced Decoupling Property: While establishing Lehner's asymmetric analogues, we uncovered that, within the RDT, they exhibit a remarkable interlaced decoupling property (see Corollary \ref{cor:cor2a0}).
  \item SDP Formulations: Finally, we discuss elegant SDP (semi-definite programming) formulations associated with the obtained asymmetric analogues (see Sections \ref{sec:handlerdA0}, \ref{sec:defnqhandlerdA0}, and \ref{sec:nqhandlerdA0}). In several concrete scenarios, we solve the underlying SDPs and compare the obtained results with the corresponding ones achieved through numerical simulations. A striking agreement between the theory and simulations is observed already for problem dimensions on the order of a few hundreds (see Tables \ref{tab:tab00}, \ref{tab:tab01}, \ref{tab:tab02}, and \ref{tab:tab1}).
\end{itemize}

\section{Symmetric indefinite scenario}
\label{sec:symmin}

We first focus on matrices $H$ from (\ref{eq:amat1a0b0}). Since these matrices are symmetric and the underlying deterministic parts of Kronecker products, $A_i$, are additionally a priori indefinite, we call scenarios considered in this section \emph{symmetric indefinite}. It should be noted that this is in stark contrast with \cite{Stojnicleh26}, where $A_i, i \ge 1$, are assumed positive semi-definite (PSD).

\subsection{Transforming Kronecker--Gaussians into plain Gaussians}
\label{sec:kronrep}

Before discussing the RDT machinery, we introduce a few technical preliminaries. First, we observe that
\begin{eqnarray}\label{eq:lbkronrep1}
\lambda_n(H)  
 = 
 \max_{\x\in\mS^{k  n}}  \x^T H\x
  =\max_{\x\in\mS^{k n}} \x^T\lp A_0\otimes I + \frac{1}{\sqrt{n}} \sum_{i=1}^{l} A_i\otimes \bar{G}_i \rp \x, 
 \end{eqnarray}
where $\mS^{p}$ stands for the unit sphere in $\mR^{p}$, i.e., $\mS^{p}=\{\x\in\mR^{p} \hspace{.03in} | \hspace{.03in} \|\x\|_2=1 \}$. Let $\bar{X}\in\mR^{n\times k}$  be $n\times k$  matrix such that 
\begin{equation}\label{eq:lbkronrep2}
\x=\mbox{vec}(\bar{X}), 
\end{equation}
where $\mbox{vec}(\cdot)$ stacks the matrix argument columns (starting from the first to the last) into a vector. Taking $\bar{X}_i,i=1\dots,k$ as the $i$-th column of $\bar{X}$, we observe
\begin{eqnarray}\label{eq:lbkronrep3}
\|\x\|_2^2 & = & \|\mbox{vec}(\bar{X})\|_2^2 = \sum_{i=1}^{k}\|\bar{X}_i\|_2^2 = \tr(\bar{X}^T\bar{X}) , 
\end{eqnarray}
and
\begin{eqnarray}\label{eq:lbkronrep4}
\lambda_n(H) &  = &  \max_{\tr(\bar{X}^T\bar{X})=1} \lp \sum_{r=1}^n\sum_{s=1}^k \lp A_0\rp_{rs}\bar{X}_r^T\bar{X}_s + \frac{1}{\sqrt{n}} \sum_{i=1}^{l} \sum_{r=1}^n\sum_{s=1}^k \lp A_i\rp_{rs}\bar{X}_r^T \bar{G}_i \bar{X}_s \rp 
\nonumber \\
  & =  & \max_{\tr(\bar{X}^T\bar{X})= 1} \lp \tr(A_0 \bar{X}^T\bar{X}) + \frac{1}{\sqrt{n}} \sum_{i=1}^{l} \tr(A_i \bar{X}^T \bar{G}_i \bar{X} ) \rp 
  \nonumber \\
  & =  & \max_{\tr(\bar{X}^T\bar{X})= 1} \lp \tr(A_0 \bar{X}^T\bar{X}) + \frac{\sqrt{2}}{\sqrt{n}} \sum_{i=1}^{l} \tr(A_i \bar{X}^T G_i \bar{X} ) \rp ,
   \end{eqnarray}
where, for $i=0,1,\dots,l$, $\lp A_i\rp_{rs}$ is the element at the intersection of the $r$-th row and the $s$-th column of $A_i$. Practically speaking, (\ref{eq:lbkronrep4}) is a convenient way to transform the Kronecker--Gaussians into the plain Gaussians.  Below, we provide a characterization of  $\lambda_n(H)$ via  Random Duality Theory (RDT).

\subsection{Lower-bounding $\lambda_n(H)$ via RDT}
\label{sec:lbxirdt}

Below is a detailed discussion of the fundamental RDT principles outlined in \cite{StojnicCSetam09,StojnicRegRndDlt10,Stojniccovmat26,Stojnicleh26}. These principles relate to:
(i) Finding the underlying optimization algebraic representation; (ii) Determining the random dual; and
(iii) Handling the random dual (for more on additional upgrades and further algorithmic implications, please see, e.g., \cite{Stojnicalgbp25,Stojnicclupsk25}). In particular, we establish the connection of each of these principles to the problems of our interest here and address the associated intricacies.

\subsubsection{Finding the underlying optimization algebraic representation} 
\label{sec:lbrandpr}

To facilitate presentation, we first consider an ortho-diagonal transformation of (\ref{eq:kronrep4}). Let $X\in\mR^{n\times k} $ and  $R_X\in\mR^{k\times k} $ be such that $X^TX=I$ and $\tr(R_X^TR_X)=1$. A change of variables, $\bar{X} = XR_X$, gives  
\begin{eqnarray}\label{eq:lbinteq1ab0}
\bar{X}^T\bar{X}  =  R_X^TR_X, \quad \mbox{and}\quad 
\frac{1}{\sqrt{n}} \tr(A_i \bar{X}^T G_i \bar{X} ) =     \frac{1}{\sqrt{n}} \tr(A_i R_X^TX^T G_i XR_X ).
\end{eqnarray}
Utilizing eigen-decomposition allows to set 
\begin{equation}\label{eq:lbinteq1ab0b0}
F_i = F_i(R_X) = R_XA_iR_X^T =    U_i D_i U_i^T ,\quad \mbox{and} \quad U_iU_i^T=U_i^TU_i=I, 
\end{equation}
where $D_i$ is the diagonal matrix with $A_i$ eigenvalues (sorted in, say, increasing order) on the main diagonal. We find it useful to also set
\begin{equation}\label{eq:lbinteq1ab0b0a0}
|F_i| =   U_i |D_i| U_i^T , 
\end{equation}
where $|D_i|$ is matrix obtained after applying $|\cdot|$ to each of the elements of $D_i$. Keeping (\ref{eq:lbinteq1ab0b0}) in mind, (\ref{eq:lbkronrep4}) can be rewritten as
\begin{eqnarray}\label{eq:lbkronrep5}
\lambda_n(H) & =  & \max_{X^TX =I,\tr(R_X^TR_X)=1} \lp \tr(A_0R_X^TR_X) +  \frac{\sqrt{2}}{\sqrt{n}} \sum_{i=1}^{l}  \tr(A_i R_X^TX^T G_i XR_X ) \rp 
\nonumber \\
  & =  & \max_{\tr(R_X^TR_X)=1} \lp \tr(A_0R_X^TR_X) +  \max_{X^TX=I}  \frac{\sqrt{2}}{\sqrt{n}} \sum_{i=1}^{l}  \tr(A_i R_X^TX^T G_i XR_X ) \rp 
\nonumber \\
  & =  & \max_{\tr(R_X^TR_X)=1} \lp \tr(  U_0D_0U_0^T ) +  \max_{X^TX=I}  \frac{\sqrt{2}}{\sqrt{n}} \sum_{i=1}^{l}  \tr( X^T G_i X  U_iD_iU_i^T  ) \rp 
 \nonumber \\
  & =  & \max_{\tr(R_X^TR_X)=1} \lp \tr(F_0) +  \max_{X^TX=I}  \frac{\sqrt{2}}{\sqrt{n}} \sum_{i=1}^{l}  \tr(X^T G_i X F_i ) \rp 
\nonumber \\
  & =  & \max_{\tr(R_X^TR_X)=1} \lp \tr(F_0) + \frac{\sqrt{2}}{\sqrt{n}} \xi \rp  ,
\end{eqnarray}
where   
\begin{eqnarray}\label{eq:lbkronrep6}
 \xi  & \triangleq &  \max_{X^TX=I}   \sum_{i=1}^{l}  \tr(X^T G_i XF_i ) ,
\end{eqnarray}
 is within RDT context called \emph{random primal}.

\subsubsection{Determining the random dual} 
\label{sec:lbranddual}

The following theorem establishes the  \emph{(complementary) random dual} analogue.

\begin{theorem}
\label{thm:lbthm1}
Let $k_1=k_2=k,l\in\mN$ be fixed and let $m=n\in\mN$ be large such that. Let $G_{i}\in\mR^{n\times n}$, $G_{i}^{(1)}\in\mR^{n\times n}$, and $G_{i}^{(2)}\in\mR^{n\times n},i=1,\dots,l$ be independent  standard normal matrices. For deterministic symmetric $A_i\in\mR^{k\times k}$ and fixed $R_X\in\mR^{k\times k}$ such that $\tr(R_X^TR_X)=1$, let $F_{i}$ and $|F_{i}|$  be as in 
(\ref{eq:lbinteq1ab0b0}) and (\ref{eq:lbinteq1ab0b0a0}). Set
\begin{eqnarray}
   \label{eq:lbthm1eq1}
 L & = &  
 \frac{1}{\sqrt{2}}\max_{X\in\mR^{n\times k}, X^TX =I} \sum_{i=1}^{l} \lp  \tr( X^T G_i^{(1)} X R_XU_i | D_i| U_i^T R_X^T) +\tr(  X^T G_i^{(2)} X R_XU_i | D_i | U_i^T R_X^T)\rp 
 \nonumber \\
& = &  \frac{1}{\sqrt{2}}\max_{X\in\mR^{n\times k}, X^TX =I} \sum_{i=1}^{l} \lp  \tr( X^T G_i^{(1)} X |F_i|) +\tr(  X^T G_i^{(2)} X |F_i| )\rp  .
\end{eqnarray}
Let $\xi$ be defined through (\ref{eq:lbinteq1ab0}), (\ref{eq:lbinteq1ab0b0}), and  (\ref{eq:lbkronrep6}) .  One then has  
\begin{eqnarray}
   \label{eq:lbthm1eq2}
\lim_{n\rightarrow \infty}  \frac{1}{\sqrt{n}}  \mE \xi   \geq  \lim_{n\rightarrow \infty} \frac{1}{\sqrt{n}} \mE L,
\end{eqnarray}
with the righthand side being the (complementary) random dual.
\end{theorem}

\begin{proof} We consider the following  two centered Gaussian processes indexed by an array $\cX = \{X\}$
  \begin{align}
\label{eq:lbmr1}
 \cG (\cX)  \triangleq   \cG (X)    &\triangleq  \sum_{i=1}^{l}  \tr( X^T G_i X  U_iD_iU_I^T  )   
 = \sum_{i=1}^{l}  \tr( X^T G_i X F_i )  
  \nonumber   \\
 \cG_l (\cX) \triangleq   \cG_l (X)  & \triangleq \frac{1}{\sqrt{2}}  \sum_{i=1}^{l} \lp  \tr( X^T G_i^{(1)} X U_i | D_i| U_i^T) +\tr(  X^T G_i^{(2)} X U_i | D_i | U_i^T )\rp 
  \nonumber   \\
& =
  \frac{1}{\sqrt{2}}  \sum_{i=1}^{l} \lp  \tr( X^T G_i^{(1)} X  |F_i|) +\tr(  X^T G_i^{(2)} X |F_i| )\rp .
  \end{align}
For two arrays $\cX^{(1)}=\{ X^{(1)}\}$ and $\cX^{(2)}=\{ X^{(2)} \}$ with $ \lp X^{(i)}\rp^T X^{(i)}= I,i=1,2$,  we further write
  \begin{eqnarray}
\label{eq:lbmr2}
\mE \cG (\cX^{(1)})\cG (\cX^{(2)})   & =  &  \sum_{i=1}^{l}   \phi_i 
\nonumber   \\
\mE \cG_l (\cX^{(1)})\cG_l (\cX^{(2)})  & =  &  \frac{1}{2}\sum_{i=1}^{l} \theta_i,
  \end{eqnarray}
  where
  \begin{eqnarray}
\label{eq:lbmr2a0}
 \phi_i & = & 
\mE \lp \tr\lp \lp X^{(1)}\rp^T G_i X^{(1)} F_i \rp 
   \tr\lp  \lp X^{(2)}  \rp^T G_i X^{(2)} F_i \rp \rp 
\nonumber   \\
\theta_i & =  &  
\mE \lp \tr\lp \lp X^{(1)}\rp^T G_i^{(1)} X^{(1)} |F_i| \rp 
   \tr\lp  \lp X^{(2)}  \rp^T G_i^{(1)} X^{(2)} |F_i| \rp \rp 
\nonumber   \\
& &    +
\mE \lp \tr\lp \lp X^{(1)}\rp^T G_i^{(2)} X^{(1)} |F_i| \rp 
   \tr\lp  \lp X^{(2)}  \rp^T G_i^{(2)} X^{(2)} |F_i| \rp \rp 
 .
  \end{eqnarray}
Moreover,
  \begin{eqnarray}
\label{eq:lbmr2a1}
 \phi_i & = & 
\mE \lp \tr\lp \lp X^{(1)}\rp^T G_i X^{(1)} F_i \rp 
   \tr\lp  \lp X^{(2)}  \rp^T G_i X^{(2)} F_i \rp \rp 
\nonumber \\
 & = &
 \tr\lp   \lp X^{(2)} F_i   \lp X^{(2)}\rp^T  \rp^T  \lp X^{(1)} F_i   \lp X^{(1)}\rp^T  \rp  \rp
\nonumber \\
 & = &
 \tr\lp   X^{(2)} F_i \lp X^{(2)}\rp^T  X^{(1)} F_i   \lp X^{(1)}\rp^T   \rp
   \nonumber \\
   & = & 
 \tr\lp  \lp X^{(2)}\rp^TX^{(1)}  F_i\lp X^{(1)}\rp^T X^{(2)}  F_i   \rp 
    \nonumber \\
   & = & 
 \tr\lp  \lp X^{(2)}\rp^TX^{(1)} U_iD_iU_i^T\lp X^{(1)}\rp^T X^{(2)}  U_iD_iU_i^T  \rp 
     \nonumber \\
   & = & 
 \tr\lp  \lp X^{(2)}\rp^TX^{(1)} U_i |D_i| \mbox{sign}(D_i) U_i^T\lp X^{(1)}\rp^T X^{(2)}  U_i \mbox{sign}(D_i)| D_i| U_i^T  \rp 
      \nonumber \\
   & = & 
 \tr\lp  U_i^T   \lp X^{(2)}\rp^TX^{(1)} U_i |D_i| \mbox{sign}(D_i) U_i^T\lp X^{(1)}\rp^T X^{(2)}  U_i \mbox{sign}(D_i)| D_i|  \rp  
      \nonumber \\
   & = & 
 \tr\lp  U_i^T   W^T U_i |D_i| \mbox{sign}(D_i) U_i^T W U_i \mbox{sign}(D_i)| D_i|  \rp 
       \nonumber \\
   & = & 
 \tr\lp  U_i^T   W^T U_i \sqrt{|D_i|} \sqrt{|D_i|} \mbox{sign}(D_i) U_i^T W U_i \mbox{sign}(D_i)
  \sqrt{|D_i|} \sqrt{|D_i|}  \rp 
       \nonumber \\
   & = & 
 \tr\lp  \sqrt{|D_i|}  U_i^T   W^T U_i \sqrt{|D_i|} \sqrt{|D_i|} \mbox{sign}(D_i) U_i^T W U_i \mbox{sign}(D_i)
  \sqrt{|D_i|}  \rp 
 ,
    \end{eqnarray}
    where
  \begin{eqnarray}
\label{eq:lbmr2a1a0}
W = \lp X^{(1)}\rp^T X^{(2)} ,
    \end{eqnarray}    
    and $\mbox{sign}(\cdot)$ is applied componentwise with the convention $\mbox{sign}(0)=0$.  We also have
 \begin{align}
\label{eq:lbmr2a2}
\theta_i 
& =  
\mE \lp \tr\lp \lp X^{(1)}\rp^T G_i^{(1)} X^{(1)} |F_i| \rp 
   \tr\lp  \lp X^{(2)}  \rp^T G_i^{(1)} X^{(2)} |F_i| \rp \rp 
\nonumber   \\
&  \hspace{.15in}   +
\mE \lp \tr\lp \lp X^{(1)}\rp^T G_i^{(2)} X^{(1)} |F_i| \rp 
   \tr\lp  \lp X^{(2)}  \rp^T G_i^{(2)} X^{(2)} |F_i| \rp \rp 
\nonumber \\
& =  
 \tr\lp   \lp X^{(2)} | F_i|   \lp X^{(2)}\rp^T  \rp^T  \lp X^{(1)} | F_i |   \lp X^{(1)}\rp^T  \rp  \rp
  +
 \tr\lp   \lp X^{(2)} | F_i |   \lp X^{(2)}\rp^T  \rp^T  \lp X^{(1)} | F_i |   \lp X^{(1)}\rp^T  \rp  \rp
\nonumber \\
& =
 \tr\lp   X^{(2)} |F_i| \lp X^{(2)}\rp^T  X^{(1)} | F_i |  \lp X^{(1)}\rp^T   \rp
+ \tr\lp   X^{(2)} | F_i| \lp X^{(2)}\rp^T  X^{(1)} | F_i |   \lp X^{(1)}\rp^T   \rp
\nonumber \\
& =
 \tr\lp  \lp X^{(2)}\rp^TX^{(1)}  |F_i|\lp X^{(1)}\rp^T X^{(2)}  |F_i |   \rp 
+
 \tr\lp  \lp X^{(2)}\rp^TX^{(1)}  | F_i |\lp X^{(1)}\rp^T X^{(2)}  |F_i|   \rp 
\nonumber \\
& =
 \tr\lp  W^T  |F_i| W |F_i |   \rp 
+
 \tr\lp  W^T  | F_i | W  |F_i|   \rp 
\nonumber \\
& =   
 \tr\lp W^T  U_i |D_i|U_i^T W U_i |D_i|U_i^T  \rp 
+
 \tr\lp  W^T  U_i |D_i|U_i^TW  U_i |D_i|U_i^T   \rp 
\nonumber \\
& =   
 \tr\lp  W^T U_i \sqrt{|D_i|} \sqrt{|D_i|} U_i^T W U_i   \sqrt{|D_i|} \sqrt{|D_i|} U_i^T  \rp 
\nonumber \\
& \hspace{.15in}
+
 \tr\lp  W^T  U_i \mbox{sign}(D_i) \sqrt{|D_i|}\sqrt{|D_i|} \mbox{sign}(D_i) U_i^T
 W  U_i \mbox{sign}(D_i) \sqrt{|D_i|}\sqrt{|D_i|} \mbox{sign}(D_i) U_i^T   \rp 
\nonumber \\
& =   
 \tr\lp  \sqrt{|D_i|}  U_i^T W^T U_i \sqrt{|D_i|} \sqrt{|D_i|} U_i^T W U_i   \sqrt{|D_i|}   \rp 
\nonumber \\
& \hspace{.15in}
+
 \tr\lp  \sqrt{|D_i|} \mbox{sign}(D_i) U_i^T  W^T  U_i \mbox{sign}(D_i) \sqrt{|D_i|}\sqrt{|D_i|} \mbox{sign}(D_i) U_i^T
 W  U_i \mbox{sign}(D_i) \sqrt{|D_i|}  \rp 
  .
   \end{align}
    From (\ref{eq:lbmr2})--(\ref{eq:lbmr2a2}), one then finds
  \begin{align}
\label{eq:lbmr5}
\mE & \cG_l (\cX^{(1)})\cG_l (\cX^{(2)} ) -   \mE  \cG (\cX^{(1)})\cG (\cX^{(2)})
 =
\nonumber
\\
 & =  \frac{1}{2} \sum_{i=1}^{l} \theta_i -   \sum_{i=1}^{l}  \phi_i   
 =  \frac{1}{2} \sum_{i=1}^{l} \lp \theta_i - 2 \phi_i \rp   
\nonumber
\\
& = 
\frac{1}{2} \sum_{i=1}^{l} 
\Bigg ( \Bigg .
 \tr\lp  \sqrt{|D_i|}  U_i^T W^T U_i \sqrt{|D_i|} \sqrt{|D_i|} U_i^T W U_i   \sqrt{|D_i|}   \rp 
\nonumber \\
& \hspace{.15in}
+
 \tr\lp  \sqrt{|D_i|} \mbox{sign}(D_i) U_i^T  W^T  U_i \mbox{sign}(D_i) \sqrt{|D_i|}\sqrt{|D_i|} \mbox{sign}(D_i) U_i^T
 W  U_i \mbox{sign}(D_i) \sqrt{|D_i|}  \rp 
\nonumber \\
& \hspace{.15in} -
 2\tr\lp  \sqrt{|D_i|}  U_i^T   W^T U_i \sqrt{|D_i|} \sqrt{|D_i|} \mbox{sign}(D_i) U_i^T W U_i \mbox{sign}(D_i)
  \sqrt{|D_i|}  \rp 
\Bigg .\Bigg )
\nonumber
\\
& = 
\frac{1}{2} \sum_{i=1}^{l} 
\Bigg ( \Bigg .
 \tr
\Bigg ( \Bigg .  
  \lp \sqrt{|D_i|} U_i^T W U_i \sqrt{|D_i|} -\sqrt{|D_i|} \mbox{sign}(D_i) U_i^T W U_i \mbox{sign}(D_i)
  \sqrt{|D_i|}   \rp^T 
\nonumber \\
& \hspace{.15in}
\times
 \lp \sqrt{|D_i|} U_i^T W U_i \sqrt{|D_i|} -\sqrt{|D_i|} \mbox{sign}(D_i) U_i^T W U_i \mbox{sign}(D_i)
  \sqrt{|D_i|}   \rp
 \Bigg .\Bigg )
 \Bigg .\Bigg )
  \nonumber \\
&  \geq   0.  
 \end{align} 
We also note  
  \begin{align}
\label{eq:lbmr5a0}
\mE \cG_l (\cX^{(1)})\cG_l (\cX^{(1)} )
-
  \mE & \cG (\cX^{(1)})\cG (\cX^{(1)})
  = 
\nonumber
\\
& = 
\frac{1}{2} \sum_{i=1}^{l} 
\Bigg ( \Bigg .
 \tr
\Bigg ( \Bigg .  
  \lp \sqrt{|D_i|} U_i^T  U_i \sqrt{|D_i|} -\sqrt{|D_i|} \mbox{sign}(D_i) U_i^T  U_i \mbox{sign}(D_i)
  \sqrt{|D_i|}   \rp^T 
\nonumber \\
& \hspace{.15in}
\times
 \lp \sqrt{|D_i|} U_i^T  U_i \sqrt{|D_i|} -\sqrt{|D_i|} \mbox{sign}(D_i) U_i^T  U_i \mbox{sign}(D_i)
  \sqrt{|D_i|}   \rp
 \Bigg .\Bigg )
 \Bigg .\Bigg )
\nonumber \\
&  =   0,
  \end{align}
  where the second equality follows since  $\lp X^{(1)}\rp^T X^{(1)}=I$, and the third since $ U_i^T  U_i =I $ and $\mbox{sign}(D_i) \mbox{sign}(D_i) =I$. 
  
  To complete the proof, we recall Theorem 1.1 from \cite{Gordon85}. The part of the theorem utilized below was introduced in \cite{Slep62} and is known as Slepian's lemma. Both results can be deduced as special cases of the concepts discussed in Corollary 3 in \cite{Stojnicgscompyx16} and Corollary 4 in \cite{Stojnicgscomp16}.

\begin{theorem}(\cite{Gordon85,Slep62})
\label{thm:Gordonpos1} Let $X_{i}$ and $Y_{i}$, $1\leq i\leq n$, be two centered Gaussian processes which satisfy the following inequalities for all choices of indices
\begin{enumerate}
\item $\mE(X_{i}^2)=\mE(Y_{i}^2)$
\item $\mE(X_{i}X_{l})\leq \mE(Y_{i}Y_{l}), i\neq l$.
\end{enumerate}
 Then
\begin{equation*}
\mE(\min_{i} X_{i})\leq \mE(\min_i Y_{i}) \quad  \Longleftrightarrow \quad \mE(\max_{i} X_{i})\geq \mE(\max_i Y_{i}).
\end{equation*}
\end{theorem}

Applying  Theorem \ref{thm:Gordonpos1} to processes $\cG(\cdot)$ and  $\cG_l(\cdot)$ with correspondence $X\leftrightarrow\cG$ and $Y\leftrightarrow\cG_l$ gives
\begin{align}\label{eq:lbmt5a1a0}
&  &\mE \max_{\cX^{(a_1)}} \cG(\cX)  & \geq \mE \max_{\cX^{(a_1)}} \cG_l(\cX)
\nonumber \\
\Longleftrightarrow & & \mE \max_{X^TX=I} 
\sum_{i=1}^{l}  \tr( X^T G_i X F_i )  
 & \geq \frac{1}{\sqrt{2}}\mE \max_{X^TX=I} \sum_{i=1}^{l}  \lp  \tr( X^T G_i^{(1)} X |F_i| )
+  \tr( X^T G_i^{(2)} X |F_i| ) \rp .
\end{align} 
Connecting further (\ref{eq:lbkronrep6}) and (\ref{eq:lbmt5a1a0}), we then also find
\begin{eqnarray}\label{eq:lbmt5a1a1}
 \lim_{n\rightarrow \infty}   \frac{1}{\sqrt{n}}  \mE  \xi  & \geq &   \lim_{n\rightarrow \infty}  \frac{1}{\sqrt{2n}} \mE \max_{X^TX=I} \sum_{i=1}^{l}  \lp  \tr( X^T G_i^{(1)} X |F_i| )
+  \tr( X^T G_i^{(2)} X |F_i| ) \rp ,
 \end{eqnarray}
which, together with (\ref{eq:lbthm1eq1}), gives  (\ref{eq:lbthm1eq2}) and completes the proof.
\end{proof}

\begin{remark}
\label{rem:rem0}
To facilitate presentation throughout the paper, we focus on expectations. However, all the key quantities discussed above trivially concentrate in the considered dimensional regimes, and the stated results automatically hold in a probabilistic sense as well.

Additionally, to ensure clarity and elegance, the results are stated in the asymptotic regime ($n\rightarrow \infty$). With additional effort, they can be adapted to hold in a non-asymptotic sense as well.
\end{remark}

\subsubsection{Handling the random dual}
\label{sec:lbhandlerd}

We first observe 
\begin{equation} 
\label{eq:lbhrd0}
   \frac{1}{\sqrt{2}}\mE \max_{X^TX=I} 
\sum_{i=1}^{l} \lp  \tr( X^T G_i^{(1)} X | F_i | )  
+
 \tr( X^T G_i^{(2)} X | F_i | )
 \rp
 =  \mE \max_{X^TX=I} \sum_{i=1}^{l}  \tr( X^T G_i X |F_i| ) ,
 \end{equation}
 where the equality holds since  $G_i$ and $\frac{1}{\sqrt{2}}\lp G_i^{(1)} +G_i^{(2)}\rp$ are statistically equivalent. Combining (\ref{eq:lbmt5a1a1}) and (\ref{eq:lbhrd0}), we further find
\begin{eqnarray}\label{eq:lbhrd0a0}
 \lim_{n\rightarrow \infty}   \frac{1}{\sqrt{n}}  \mE  \xi  
 & \geq &   \lim_{n\rightarrow \infty}  \frac{1}{\sqrt{2n}} \mE \max_{X^TX=I} \sum_{i=1}^{l}  \lp  \tr( X^T G_i^{(1)} X |F_i| )
+  \tr( X^T G_i^{(2)} X |F_i| ) \rp 
\nonumber \\
 & = &   \lim_{n\rightarrow \infty}  \frac{1}{\sqrt{n}} \mE \max_{X^TX=I} \sum_{i=1}^{l}  \tr( X^T G_i X |F_i| ) 
 \nonumber \\
 & = &   \lim_{n\rightarrow \infty}  \frac{1}{\sqrt{n}} \mE \xi_l ,
 \end{eqnarray}
 where
\begin{eqnarray}\label{eq:lbhrd0a0a0}
\xi_l \triangleq  \max_{X^TX=I} \sum_{i=1}^{l}  \tr( X^T G_i X |F_i| ).
\end{eqnarray} 
Since $|F_i|\succeq 0,i=1,\dots,k$, one can then utilize results of \cite{Stojnicleh26}. In particular, we have the following theorem.

\begin{theorem}
\label{thm:lbthm1a0} Assume the setup of Theorem \ref{thm:lbthm1} and
let $G_{i}^{(3)}\in\mR^{k\times n},i=1,\dots,l$ be a random matrix independent of all other randomness and comprised of independent standard normals. Set
\begin{eqnarray}
   \label{eq:lbthm1eq1a0}
 L_l & = &  
   \max_{X\in\mR^{n\times k}, X^TX =I} \sum_{i=1}^{l} \sqrt{2} \tr\lp G_i^{(3)} X |F_i| \rp   .
\end{eqnarray}
Let $\xi_l$ be as in (\ref{eq:lbhrd0a0a0}).  One then has  
\begin{eqnarray}
   \label{eq:lbthm1eq2}
\lim_{n\rightarrow \infty}  \frac{1}{\sqrt{n}}  \mE \xi  
 \geq 
 \lim_{n\rightarrow \infty}  \frac{1}{\sqrt{n}}  \mE \xi_l
  = \lim_{n\rightarrow \infty} \frac{1}{\sqrt{n}} \mE L_l.
\end{eqnarray}
 \end{theorem}

\begin{proof}
The inequality in (\ref{eq:lbthm1eq2}) follows from (\ref{eq:lbhrd0a0}). On the other hand, the equality
follows automatically through Theorem 1, Corollary 1, and Theorem 5 from \cite{Stojnicleh26}, after one recognizes that $|F_i|\succeq 0$.
\end{proof}

Following derivation (26)-(35) from \cite{Stojnicleh26} allows to write
 \begin{eqnarray}
   \label{eq:lbhrd9}
\lim_{n\rightarrow\infty} \frac{1}{\sqrt{n}} \mE L_l   
& = &
\min_{\Gamma=\Gamma^T} \bar{L}_l,
\end{eqnarray}
where
\begin{eqnarray}
   \label{eq:lbhrd9b0}
 \bar{L}_l \triangleq    
\frac{1}{\sqrt{2}} \lp \tr\lp\lp \sum_{i=1}^{l}
| F_i | | F_i |^T \rp\Gamma^{-1} \rp + \tr(\Gamma)\rp
=
\frac{1}{\sqrt{2}} \lp \tr\lp\lp \sum_{i=1}^{l}
 F_i  F_i ^T \rp\Gamma^{-1} \rp + \tr(\Gamma)\rp,
\end{eqnarray}
and the last equality holds since
\begin{eqnarray}
   \label{eq:lbhrd9b0a0}
| F_i | | F_i |^T = U_i |D_i|U_i^TU_i |D_i|U_i^T =U_i |D_i|^2U_i^T = U_i D_i^2 U_i^T =
U_i D_i D_iU_i^T = U_i D_i U_iU_i^T  D_iU_i^T =F_iF_i^T.
\end{eqnarray}
A combination of  (\ref{eq:lbkronrep5}), (\ref{eq:lbthm1eq2}), and  (\ref{eq:lbhrd9}) gives
\begin{eqnarray}\label{eq:lbhrd9a0}
 \lim_{n\rightarrow \infty}\mE\lambda_n(H) & =  & \max_{\tr(R_X^TR_X)=1} \lp \tr(F_0) +  \sqrt{2}\lim_{n\rightarrow \infty} \frac{1}{\sqrt{n}} \mE \xi \rp  
\nonumber \\
& \geq & 
 \max_{\tr(R_X^TR_X)=1} \lp \tr(F_0) +  \min_{\Gamma=\Gamma^T} \lp \tr\lp\lp \sum_{i=1}^{l}
F_iF_i^T \rp \Gamma^{-1} \rp + \tr(\Gamma) \rp \rp .
\end{eqnarray}
Solving the inner minimization gives optimal $\Gamma$
\begin{eqnarray}\label{eq:lbhrd9a0a0}
\tilde{\Gamma} =  \sqrt{\sum_{i=1}^{l} F_iF_i^T} , 
\end{eqnarray} 
and
\begin{eqnarray}\label{eq:lbhrd9a1}
 \lim_{n\rightarrow \infty}\mE\lambda_n(H)  
& \geq & 
  \max_{\tr(R_X^TR_X)=1} \lp \tr(A_0R_X^TR_X) +  2 \tr \sqrt{ \sum_{i=1}^{l}
( R_XA_iR_X^T )( R_XA_iR_X^T )^T } \rp  
\nonumber \\
&  = & 
 \max_{\tr(R_X^TR_X)=1} \lp \tr(A_0R_X^TR_X) +  2\tr\sqrt{ \sum_{i=1}^{l}
 A_iR_X^TR_X A_iR_X^TR_X } \rp 
 \nonumber \\
&  = & 
 \max_{S=S^T\succeq 0, \tr(S)=1} \lp \tr(A_0S) +  2 \tr\sqrt{ \sum_{i=1}^{l}
 (A_i S)^2 } \rp 
 \nonumber \\
&  = & 
\rho_n ,
\end{eqnarray}
where the last equality is obtained in Proposition 3.3 in \cite{CollYama26} (see, also (38) in \cite{Stojnicleh26}, Theorem 1.9 in \cite{Kunisky26}, and \cite{bandfree23,bandfree24} for related results).  Since the above recovers $\rho_n$, one has that the Lehner formula is effectively a lower bound for spectral edges in the symmetric indefinite scenario (as stated earlier, based on spectral characterizations \cite{HaagThor05,CollinsGP22,Schultz05}, this bound is actually tight; for PSD scenario tightness is proven via RDT in \cite{Stojnicleh26} as well).

\section{Asymmetric scenario}
\label{sec:asm}

In this section, we focus on the asymmetric scenario. This means that instead of $H$ from (\ref{eq:amat1a0b0}), we consider 
\begin{equation}\label{eq:asmamat1a0b0}
 H = A_0\otimes I + \frac{1}{\sqrt{n}} \sum_{i=1}^{l} A_i\otimes G_i, 
 \end{equation}
where $G_i$ are comprised of independent standard normals. In the asymmetric scenario, matrices are not necessarily square. To account for this, we take $A_i \in \mathbb{R}^{k_1 \times k_2}$ and formally write $G_i \in \mathbb{R}^{m \times n}$. However, for (\ref{eq:asmamat1a0b0}) to hold, we assume $m=n$ for the time being (specifically whenever $A_0 \neq 0$). Scenarios where $m \neq n$ will be explicitly emphasized throughout the presentation. Additionally, since $H$ is now asymmetric, we will consider its leading singular value unless otherwise stated.

Before proceeding, it is useful to note that simple algebraic maneuvering allows for an automatic translation of symmetric results into asymmetric ones. For example, expanding $A_i$ to $\bar{A}_i$
 \begin{equation}\label{eq:symmasymm1}
  \bar{A}_i = \begin{bmatrix}
        0 & A_i & A_i^T & 0 
      \end{bmatrix},
\end{equation} 
 and properly adjusting $G_i$ for $m=n$, one obtains symmetric scenarios to which the results from the previous section and \cite{Stojnicleh26} immediately apply. However, our focus in this work is to explore the degree to which the RDT methodology extends. Along those lines, our primary interest is not just determining asymmetric analogues, but rather checking whether that can be accomplished via the RDT.

\subsection{Plain Gaussians instead of Kronecker--Gaussians}
\label{sec:kronrep}

Following the above guidelines, we proceed by relying on the RDT machinery. To ensure presentation smoothness, we try to parallel the exposition of the previous sections and start by introducing a few technical preliminaries.

Analogously to (\ref{eq:amat1a0b1}), let function $\s_i(\cdot)$ determine the $i$-th smallest singular value of its argument. One then  has
 \begin{equation}\label{eq:asmamat1a0b1}
 \s_1(H) \leq  \s_2(H) \leq \dots  \leq  \s_n(H) , 
 \end{equation}
 where $ \s_1(H)$ and  $ \s_n(H)$ are the smallest and the largest singular values of $H$, respectively. We then observe
\begin{eqnarray}\label{eq:kronrep1}
s_n(H) & = & \sqrt{\lambda_n(H^TH)}  
\nonumber \\
& = &
 \sqrt{\max_{\x\in\mS^{k_2  n}}  \x H^TH\x}
\nonumber \\
& = &
 \max_{\x\in\mS^{k_2 n}}  \|H\x \|_2
\nonumber \\
& = &
 \max_{\x\in\mS^{k_2 n}}  \max_{\y\in\mS^{k_1 m}}  \y^T H\x 
\nonumber \\
& = & 
 \max_{\x\in\mS^{k_2 n},\y\in\mS^{k_1 m}} \y^T\lp A_0\otimes I + \frac{1}{\sqrt{n}} \sum_{i=1}^{l} A_i\otimes G_i \rp \x .
 \end{eqnarray}
 Let $\bar{X}\in\mR^{n\times k_2}$ and $\bar{Y}\in\mR^{m\times k_1}$ be $n\times k_2$ and $m\times k_1$  matrices such that 
\begin{equation}\label{eq:kronrep2}
\x=\mbox{vec}(\bar{X}) \quad \mbox{and} \quad 
\y=\mbox{vec}(\bar{Y}). 
\end{equation}
 Also, let $\bar{X}_i,i=1\dots,k_2$, $\bar{Y}_i,i=1\dots,k_1$ be the $i$-th columns of $\bar{X}$ and $\bar{Y}$, respectively. We then first have
\begin{eqnarray}\label{eq:kronrep3}
\|\x\|_2^2 & = & \|\mbox{vec}(\bar{X})\|_2^2 = \sum_{i=1}^{k_2}\|\bar{X}_i\|_2^2 = \tr(\bar{X}^T\bar{X})
\nonumber \\
\|\y\|_2^2 & = & \|\mbox{vec}(\bar{Y})\|_2^2 = \sum_{i=1}^{k_1}\|\bar{Y}_i\|_2^2 = \tr(\bar{Y}^T\bar{Y}) , 
\end{eqnarray}
and
\begin{eqnarray}\label{eq:kronrep4}
\s_n(H) &  = &  \max_{\tr(\bar{X}^T\bar{X})=\tr(\bar{Y}^T\bar{Y})=1} \lp \sum_{r=1}^{k_1}\sum_{s=1}^{k_2} \lp A_0\rp_{rs}\bar{Y}_r^T\bar{X}_s + \frac{1}{\sqrt{n}} \sum_{i=1}^{l} \sum_{r=1}^{k_1}\sum_{s=1}^{k_2} \lp A_i\rp_{rs}\bar{Y}_r^T G_i \bar{X}_s \rp 
\nonumber \\
  & =  & \max_{\tr(\bar{X}^T\bar{X})=\tr(\bar{Y}^T\bar{Y})=1} \lp \tr(A_0^T \bar{Y}^T\bar{X}) + \frac{1}{\sqrt{n}} \sum_{i=1}^{l} \tr(A_i^T \bar{Y}^T G_i \bar{X} ) \rp ,
   \end{eqnarray}
where, for $i=0,1,\dots,l$, $\lp A_i\rp_{rs}$ is the element at the intersection of the $r$-th row and the $s$-th column of $A_i$. (\ref{eq:kronrep4}) is practically a convenient representation of the Kronecker--Gaussians via plain Gaussians.  Below, we focus on an RDT characterization of  $\s_n(H)$.

\subsection{Lower-bounding $\s_n(H)$ via RDT}
\label{sec:xirdt}

Mimicking the presentation of Section \ref{sec:xirdt}, we will separately discuss the key RDT principles and how they relate to the problems of interest here.

\subsubsection{Finding underlying optimization algebraic representation} 
\label{sec:randpr}

Utilizing  ortho-diagonal transformations of (\ref{eq:kronrep4}), we let $X\in\mR^{n\times k_2} $ and  $R_X\in\mR^{k_2\times k_2} $ be such that $X^TX=I$ and $\tr(R_X^TR_X)=1$. Analogously, we consider $Y\in\mR^{m\times k_1} $ and  $R_Y\in\mR^{k_1\times k_1} $ such that $Y^TY=I$ and $\tr(R_Y^TR_Y)=1$. Changing variables, $\bar{X} = XR_X$ and $\bar{Y} = YR_Y$, we obtain  
\begin{eqnarray}\label{eq:inteq1ab0}
\bar{X}^T\bar{X}  =  R_X^TR_X, \quad\quad 
\bar{Y}^T\bar{Y} = R_Y^TR_Y, \quad \mbox{and}\quad 
\frac{1}{\sqrt{n}} \tr(A_i^T \bar{Y}^T G_i \bar{X} ) =     \frac{1}{\sqrt{n}} \tr(A_i^T R_Y^TY^T G_i XR_X ).
\end{eqnarray}
After adjusting $F_i$ definition from (\ref{eq:lbinteq1ab0b0}) so that it fits the asymmetric scenario, we have
\begin{equation}\label{eq:inteq1ab0b0}
F_i = F_i(R_X,R_Y) = R_XA_i^TR_Y^T =V_iS_iU_i^T , \quad \mbox{and} \quad V_i^TV_i=U_i^TU_i = I,
\end{equation}
where $S_i$ is a diagonal matrix with singular values of $F_i$ on the main diagonal. We also find it useful to set
\begin{equation}\label{eq:inteq1ab0b0a0}
F_i = F_{i,X}F_{i,Y}^T  \quad \mbox{with} \quad 
F_{i,X} = V_i\sqrt{S_i}\quad \mbox{and}\quad
F_{i,Y} = U_i\sqrt{S_i}.
\end{equation}
We then rewrite (\ref{eq:kronrep4}) as
\begin{eqnarray}\label{eq:kronrep5}
\s_n(H) & =  & \max_{X^TX=Y^TY=I,\tr(R_X^TR_X)=\tr(R_Y^TR_Y)=1} \lp \tr(A_0^T R_Y^TY^TX R_X) +  \frac{1}{\sqrt{n}} \sum_{i=1}^{l}  \tr(A_i^T R_Y^TY^T G_i XR_X ) \rp 
 \nonumber \\
  & =  & \max_{X^TX=Y^TY=I,\tr(R_X^TR_X)=\tr(R_Y^TR_Y)=1} \lp \tr(A_0^TR_Y^T Y^TX  R_X) +  \frac{1}{\sqrt{n}} \sum_{i=1}^{l}  \tr( Y^T G_i X F_i) \rp 
 \nonumber \\
  & =  & \max_{\tr(R_X^TR_X)=\tr(R_Y^TR_Y)=1} \lp  \frac{1}{\sqrt{n}}\xi_a \rp  ,
\end{eqnarray}
where   
\begin{eqnarray}\label{eq:kronrep6}
\varphi = \sqrt{n} \tr(A_0^TR_Y^T Y^TX  R_X)\quad \mbox{and} \quad  \xi_a  & \triangleq &  \max_{X^TX=Y^TY=I} \lp \varphi +   \sum_{i=1}^{l}  \tr(Y^T G_i X F_i ) \rp,
\end{eqnarray}
 is the so-called \emph{random primal}.

\subsubsection{Determining the random dual} 
\label{sec:randdual}

The following theorem establishes the \emph{(complementary) random dual} analogue.

\begin{theorem}
\label{thm:thm1}
Let $k_1,k_2,l\in\mN$ be fixed and let $n,m\in\mN$ be large such that $\alpha=\lim_{n\rightarrow \infty} \frac{m}{n}$ remains constant. let $G_{i}\in\mR^{m\times n}$, $G_{i}^{(1)}\in\mR^{k_1\times m}$, and $G_{i}^{(2)}\in\mR^{k_2\times n},i=1,\dots,l$ be independent  standard normal matrices. For deterministic $A_i\in\mR^{k_1\times k_2}$  and fixed $R_X\in\mR^{k_2\times k_2}$ and $R_Y\in\mR^{k_1\times k_1}$ such that $\tr(R_X^TR_X)=\tr(R_Y^TR_Y)=1$, let $F_{i}$, $F_{i,X}$, and $F_{i,Y}$  be as in 
(\ref{eq:inteq1ab0b0}) and (\ref{eq:inteq1ab0b0a0}). Set $\varphi = \sqrt{n} \tr(A_0^TR_Y^T Y^TX  R_X) $ and 
\begin{equation}
   \label{eq:thm1eq1}
 L_a  =   
\max_{X^TX=Y^TY=I} 
 \lp \varphi +
 \frac{1}{\sqrt{2}} 
 \sum_{i=1}^{l} \lp  \tr(Y^T G_i^{(1)} Y F_{i,Y} F_{i,Y}^T ) +\tr(X^T G_i^{(2)} X F_{i,X} F_{i,X}^T ) \rp 
  \rp.
\end{equation}
Let $\xi_a$ be defined through (\ref{eq:inteq1ab0}), (\ref{eq:inteq1ab0b0}), and  (\ref{eq:kronrep6}).  One then has  
\begin{eqnarray}
   \label{eq:thm1eq2}
\lim_{n\rightarrow \infty}  \frac{1}{\sqrt{n}}  \mE \xi_a   \geq  \lim_{n\rightarrow \infty} \frac{1}{\sqrt{n}} \mE L_a,
\end{eqnarray}
with the righthand side being the (complementary) random dual.
\end{theorem}

\begin{proof} We consider two centered Gaussian processes indexed by an array $\cX = \{X,Y\}$
  \begin{align}
\label{eq:mr1}
 \cG (\cX)  & =   \cG (X,Y)   =  \sum_{i=1}^{l}  \tr( Y^T G_i X F_i )    \nonumber   \\
 \cG_l (\cX) & =   \cG_l (X,Y)   =  \frac{1}{\sqrt{2}}  \sum_{i=1}^{l}   \lp  \tr( Y^T G_i^{(1)} Y F_{i,Y} F_{i,Y}^T ) + \tr( X^T G_i^{(2)} X F_{i,X} F_{i,X}^T )\rp  .
  \end{align}
 For two arrays $\cX^{(1)}=\{ X^{(1)}, Y^{(1)}\}$ and $\cX^{(2)}=\{ X^{(2)}, Y^{(2)}\}$ with $ \lp X^{(i)}\rp^T X^{(i)}= \lp Y^{(i)}\rp^T Y^{(i)}=I,i=1,2$,  we have
  \begin{eqnarray}
\label{eq:mr2}
\mE \cG (\cX^{(1)})\cG (\cX^{(2)})   & =  &  \sum_{i=1}^{l}  \phi_i  
\nonumber   \\
\mE \cG_l (\cX^{(1)})\cG_l (\cX^{(2)})  & =  &  \frac{1}{2} \sum_{i=1}^{l} \theta_i,
  \end{eqnarray}
  where
  \begin{eqnarray}
\label{eq:mr2a0}
 \phi_i & = & 
\mE \lp \tr\lp \lp Y^{(1)}\rp^T G_i X^{(1)} F_i \rp 
   \tr\lp  \lp Y^{(2)}  \rp^T G_i X^{(2)} F_i \rp \rp 
\nonumber   \\
\theta_i & =  & 
 \mE \lp \tr\lp \lp Y^{(1)}\rp^T G_i Y^{(1)} F_{i,Y} F_{i,Y}^T \rp 
   \tr\lp  \lp Y^{(2)}  \rp^T G_i Y^{(2)} F_{i,Y} F_{i,Y}^T \rp \rp 
\nonumber   \\
& & +
 \mE \lp \tr\lp \lp X^{(1)}\rp^T G_i X^{(1)} F_{i,X} F_{i,X}^T \rp 
   \tr\lp  \lp X^{(2)}  \rp^T G_i X^{(2)} F_{i,X} F_{i,X}^T \rp \rp 
 .
  \end{eqnarray}
We further find
  \begin{eqnarray}
\label{eq:mr2a1}
 \phi_i & = & 
\mE \lp \tr\lp \lp Y^{(1)}\rp^T G_i X^{(1)} F_i \rp 
   \tr\lp  \lp Y^{(2)}  \rp^T G_i X^{(2)} F_i \rp \rp 
\nonumber \\
 & = &
 \tr\lp   \lp X^{(2)} F_i   \lp Y^{(2)}\rp^T  \rp^T  \lp X^{(1)} F_i   \lp Y^{(1)}\rp^T  \rp  \rp
\nonumber \\
 & = &
 \tr\lp   Y^{(2)} F_i^T \lp X^{(2)}\rp^T  X^{(1)} F_i   \lp Y^{(1)}\rp^T   \rp
   \nonumber \\
   & = & 
 \tr\lp  \lp X^{(2)}\rp^TX^{(1)}  F_i\lp Y^{(1)}\rp^TY^{(2)}  F_i^T   \rp 
   \nonumber \\
   & = & 
 \tr\lp  \lp X^{(2)}\rp^TX^{(1)}  F_{i,X}F_{i,Y}^T\lp Y^{(1)}\rp^TY^{(2)}  F_{i,Y}F_{i,X}^T   \rp 
    \nonumber \\
   & = & 
 \tr\lp  F_{i,X}^T \lp X^{(2)}\rp^TX^{(1)}  F_{i,X}F_{i,Y}^T\lp Y^{(1)}\rp^TY^{(2)}  F_{i,Y}   \rp 
     \nonumber \\
   & = & 
 \tr\lp  F_{i,X}^T W_X F_{i,X}F_{i,Y}^T W_Y^T F_{i,Y}   \rp ,
    \end{eqnarray}
    where
  \begin{eqnarray}
\label{eq:mr2a1a0}    
    W_X = \lp X^{(2)}\rp^TX^{(1)}  \quad \mbox{and}\quad
     W_Y = \lp Y^{(2)}\rp^TY^{(1)} .
    \end{eqnarray}
Also, one has
 \begin{eqnarray}
\label{eq:mr2a2}
\theta_i & =  &  
 \mE \lp \tr\lp \lp Y^{(1)}\rp^T G_i Y^{(1)} F_{i,Y}F_{i,Y}^T \rp 
   \tr\lp  \lp Y^{(2)}  \rp^T G_i Y^{(2)} F_{i,Y}F_{i,Y}^T \rp \rp 
\nonumber   \\
& & +
 \mE \lp \tr\lp \lp X^{(1)}\rp^T G_i X^{(1)} F_{i,X}F_{i,X}^T \rp 
   \tr\lp  \lp X^{(2)}  \rp^T G_i X^{(2)} F_{i,X}F_{i,X}^T \rp \rp 
\nonumber \\
&  = & 
 \tr\lp   \lp Y^{(2)} F_{i,Y}F_{i,Y}^T  \lp Y^{(2)}\rp^T  \rp^T  \lp Y^{(1)} F_{i,Y}F_{i,Y}^T  \lp Y^{(1)}\rp^T  \rp  \rp
\nonumber \\
& & +
 \tr\lp   \lp X^{(2)} F_{i,X}F_{i,X}^T   \lp X^{(2)}\rp^T  \rp^T  \lp X^{(1)} F_{i,X}F_{i,X}^T   \lp X^{(1)}\rp^T  \rp  \rp
\nonumber \\
&  = &    \tr\lp   Y^{(2)} F_{i,Y}F_{i,Y}^T \lp Y^{(2)}\rp^T  Y^{(1)} 
F_{i,Y}F_{i,Y}^T
   \lp Y^{(1)}\rp^T   \rp
\nonumber \\
& & +
\tr\lp   X^{(2)}F_{i,X}F_{i,X}^T \lp X^{(2)}\rp^T  X^{(1)} F_{i,X}F_{i,X}^T   \lp X^{(1)}\rp^T   \rp
\nonumber \\
&  = &  
 \tr\lp  F_{i,Y}^T \lp Y^{(2)}\rp^T Y^{(1)}  F_{i,Y} F_{i,Y}^T \lp Y^{(1)}\rp^TY^{(2)}  F_{i,Y}   \rp 
\nonumber \\
& &  +
  \tr\lp  F_{i,X}^T \lp X^{(2)}\rp^TX^{(1)}  F_{i,X} F_{i,X}^T \lp X^{(1)}\rp^TX^{(2)} F_{i,X}   \rp 
  \nonumber \\
&  = &  
 \tr\lp  F_{i,Y}^T W_Y  F_{i,Y} F_{i,Y}^T W_Y^T  F_{i,Y}   \rp 
  +
  \tr\lp  F_{i,X}^T W_X  F_{i,X} F_{i,X}^T W_X^T F_{i,X}   \rp.
   \end{eqnarray}
    From (\ref{eq:mr2})--(\ref{eq:mr2a2}), one then finds
  \begin{align}
\label{eq:mr5}
& \mE  \cG_l (\cX^{(1)})\cG_l (\cX^{(2)} )
-
  \mE  \cG (\cX^{(1)})\cG (\cX^{(2)})
 =
\nonumber
\\
 & =    \frac{1}{2}\sum_{i=1}^{l} \theta_i -  \sum_{i=1}^{l} \phi_i
=    \frac{1}{2}  \sum_{i=1}^{l} \lp \theta_i -2 \phi_i \rp 
\nonumber \\
\\
& = 
  \frac{1}{2}  \sum_{i=1}^{l} 
 \lp  \tr\lp  F_{i,Y}^T W_Y  F_{i,Y} F_{i,Y}^T W_Y^T  F_{i,Y}   \rp 
  +
  \tr\lp  F_{i,X}^T W_X  F_{i,X} F_{i,X}^T W_X^T F_{i,X}   \rp
  -
2 \tr\lp  F_{i,X}^T W_X F_{i,X}F_{i,Y}^T W_Y^T F_{i,Y}   \rp 
\rp
  \nonumber
\\
 & = 
  \frac{1}{2}  \sum_{i=1}^{l} 
 \tr\lp   
 \lp F_{i,X}^T W_X  F_{i,X}  - F_{i,Y}^T W_Y  F_{i,Y}    \rp^T
\lp F_{i,X}^T W_X  F_{i,X}  - F_{i,Y}^T W_Y  F_{i,Y}    \rp
\rp
  \nonumber \\
&  \geq   0.  
 \end{align}
We also note  
  \begin{align}
\label{eq:mr5a0}
\mE \cG_l (\cX^{(1)})\cG_l (\cX^{(1)} )
-
  \mE & \cG (\cX^{(1)})\cG (\cX^{(1)})
    = 
  \frac{1}{2}  \sum_{i=1}^{l} 
 \tr\lp   
 \lp F_{i,X}^T   F_{i,X}  - F_{i,Y}^T   F_{i,Y}    \rp^T
\lp F_{i,X}^T   F_{i,X}  - F_{i,Y}^T   F_{i,Y}    \rp
\rp
  =   0,
  \end{align}
  where the first equality follows since  $\lp X^{(1)}\rp^T X^{(1)}=I$ and $\lp Y^{(1)}\rp^T Y^{(1)}=I$. The last equality follows since $F_{i,X}^T F_{i,X}  = \sqrt{S_i}V_i^T V_i\sqrt{S_i}=S_i=\sqrt{S_i}U_i^T U_i\sqrt{S_i}=F_{i,Y}^T F_{i,Y}$.
  
Relying on concentrations and applying a probabilistic variant of Theorem \ref{thm:Gordonpos1} (or alternatively results of \cite{Stojnicgscompyx16,Stojnicgscomp16}) to processes $\cG(\cdot)$ and  $\cG_l(\cdot)$ with correspondence $X\leftrightarrow\cG$ and $Y\leftrightarrow\cG_l$, we find that
\begin{equation}\label{eq:mt5a1a0}
\mE \max_{\cX^{(a_1)}} \cG(\cX)   \geq \mE \max_{\cX^{(a_1)}} \cG_l(\cX) 
\end{equation}
implies 
\begin{multline}\label{eq:mt5a1a0a0}
 \mE \max_{\substack{X^TX=I\\Y^TY=I }}
\lp \varphi +\sum_{i=1}^{l}  \tr( Y^T G_i X F_i )  \rp
\\
  \geq 
 \mE \max_{\substack{X^TX=I\\Y^TY=I }} 
\lp \varphi + \frac{1}{\sqrt{2}}  \sum_{i=1}^{l}   \lp  \tr( Y^T G_i^{(1)} Y F_{i,Y} F_{i,Y}^T ) + \tr( X^T G_i^{(2)} X F_{i,X} F_{i,X}^T )\rp \rp
  .
 \end{multline}
Connecting further (\ref{eq:kronrep6}) and (\ref{eq:mt5a1a0}), we then also find
\begin{equation}\label{eq:mt5a1a1}
 \lim_{n\rightarrow \infty}   \frac{1}{\sqrt{n}}  \mE  \xi_a   \geq    \lim_{n\rightarrow \infty}  \frac{1}{\sqrt{n}}  \mE \max_{X^TX=Y^TY=I} 
\lp \varphi + \frac{1}{\sqrt{2}}  \sum_{i=1}^{l}   \lp  \tr( Y^T G_i^{(1)} Y F_{i,Y} F_{i,Y}^T ) + \tr( X^T G_i^{(2)} X F_{i,X} F_{i,X}^T )\rp \rp
,
 \end{equation}
which, together with (\ref{eq:thm1eq1}), gives  (\ref{eq:thm1eq2}) and completes the proof.
\end{proof}

\subsubsection{Handling the random dual}
\label{sec:handlerd}

We separately discuss three scenarios: (i) $A_0=0$; (ii) Deformed $A_0=0$; and (iii) $A_0\neq0$.

\subsubsubsection{Centered scenario ($A_0=0$)}
\label{sec:handlerdA0}

In the centered scenario we have $A_0=0$ and the deterministic shift by $A_0\otimes I$ is no longer present. Since $A_0=0$ we also have $\varphi = 0$. This then allows to  observe 
 \begin{eqnarray}\label{eq:hrd0a0}
 \lim_{n\rightarrow \infty}   \frac{1}{\sqrt{n}}  \mE  \xi_a  
 & \geq &    \lim_{n\rightarrow \infty}  \frac{1}{\sqrt{n}}  \mE \max_{X^TX=Y^TY=I} 
\frac{1}{\sqrt{2}}  \sum_{i=1}^{l}   \lp  \tr( Y^T G_i^{(1)} Y F_{i,Y} F_{i,Y}^T ) + \tr( X^T G_i^{(2)} X F_{i,X} F_{i,X}^T )\rp 
\nonumber \\
& \geq &    \lim_{n\rightarrow \infty}  \frac{1}{\sqrt{n}}  \mE \max_{Y^TY=I} 
\frac{1}{\sqrt{2}}  \sum_{i=1}^{l}    \tr( Y^T G_i^{(1)} Y F_{i,Y} F_{i,Y}^T ) 
\nonumber \\
& & + 
\lim_{n\rightarrow \infty}  \frac{1}{\sqrt{n}}  \mE \max_{X^TX = I} 
\frac{1}{\sqrt{2}}  \sum_{i=1}^{l}  \tr( X^T G_i^{(2)} X F_{i,X} F_{i,X}^T ) 
  \nonumber \\
 & = &   \lim_{n\rightarrow \infty}  \frac{1}{\sqrt{n}} \mE \xi_{l,Y} 
 + \lim_{n\rightarrow \infty}  \frac{1}{\sqrt{n}} \mE \xi_{l,X},
 \end{eqnarray}
 where
\begin{eqnarray}\label{eq:hrd0a0a0}
\xi_{l,Y} \triangleq  \max_{Y^TY=I} 
\frac{1}{\sqrt{2}}  \sum_{i=1}^{l}    \tr( Y^T G_i^{(1)} Y F_{i,Y} F_{i,Y}^T ) 
\quad \mbox{and}\quad 
\xi_{l,X} \triangleq  \max_{X^TX=I} 
\frac{1}{\sqrt{2}}  \sum_{i=1}^{l}    \tr( X^T G_i^{(2)} X F_{i,X} F_{i,X}^T ) 
.
\end{eqnarray} 
Since $F_{i,X}F_{i,X}^T\succeq 0$ and $F_{i,Y}F_{i,Y}^T\succeq 0,i=1,\dots,k$, we can again utilize results of \cite{Stojnicleh26}. In particular, we have the following analogue to Theorem \ref{thm:lbthm1a0}.

\begin{theorem}
\label{thm:thm1a0} Assume the setup of Theorem \ref{thm:thm1} with $A_0=0$. For $m,n\in\mN$ such that $\alpha=\lim_{n\rightarrow \infty} \frac{m}{n}$, let $G_{i}^{(4)}\in\mR^{k_1\times m}$ and $G_{i}^{(5)}\in\mR^{k_2\times n},i=1,\dots,l$ be two matrices comprised of independent standard normals. Moreover, let $G_{i}^{(4)}$ and $G_{i}^{(5)}$ be independent of all other randomness. Set
\begin{eqnarray}
   \label{eq:thm1a0eq1a0}
 L_{l,Y} & = &  
   \max_{Y\in\mR^{m\times k_1}, Y^TY =I} \sum_{i=1}^{l}  \tr\lp G_i^{(4)} Y F_{i,Y}F_{i,Y}^T \rp   \nonumber \\
    L_{l,X} & = &  
   \max_{X\in\mR^{n\times k_2}, X^TX =I} \sum_{i=1}^{l}  \tr\lp G_i^{(5)} X F_{i,X}F_{i,X}^T \rp      .
\end{eqnarray}
Let $\xi_{l,Y}$ and $\xi_{l,X}$ be as in (\ref{eq:hrd0a0a0}).  One then has  
\begin{eqnarray}
   \label{eq:thm1a0eq2}
\lim_{n\rightarrow \infty}  \frac{1}{\sqrt{n}}  \mE \xi_a  
 \geq 
 \lim_{n\rightarrow \infty}  \frac{1}{\sqrt{n}}  \mE \xi_{l,Y}
+
 \lim_{n\rightarrow \infty}  \frac{1}{\sqrt{n}}  \mE \xi_{l,X}
  = \lim_{n\rightarrow \infty} \frac{1}{\sqrt{n}} \mE L_{l,Y}
  +
  \lim_{n\rightarrow \infty} \frac{1}{\sqrt{n}} \mE L_{l,X}.
\end{eqnarray}
 \end{theorem}

\begin{proof}
The inequality  follows from (\ref{eq:hrd0a0}). The equality follows from Theorem \ref{thm:lbthm1} (and ultimately \cite{Stojnicleh26}), after one recognizes that $F_{i,X}F_{i,X}^T\succeq 0$ and $F_{i,Y}F_{i,Y}^T\succeq 0$.
\end{proof}

\begin{remark}
  \label{rem:rem1} Since $A_0=0$, one can now allow $m\neq n$ (i.e., $\alpha\neq 1$) scenario as well.
\end{remark}

Following once again derivation (26)-(35) from \cite{Stojnicleh26}, we can write
 \begin{eqnarray}
   \label{eq:hrd9}
\lim_{n\rightarrow\infty} \frac{1}{\sqrt{n}} \mE L_{l,Y}   
 = 
\min_{\Gamma_Y=\Gamma_Y^T} \bar{L}_{l,Y}\quad
\mbox{and}\quad
\lim_{n\rightarrow\infty} \frac{1}{\sqrt{n}} \mE L_{l,X}   
 = 
\min_{\Gamma_X=\Gamma_X^T} \bar{L}_{l,X}
\end{eqnarray}
where
\begin{eqnarray}
   \label{eq:hrd9b0}
 \bar{L}_{l,Y} & \triangleq  &  
\frac{1}{2} \lp \alpha \tr\lp\lp \sum_{i=1}^{l}
\lp F_{i,Y} F_{i,Y}^T \rp^2\rp\Gamma_Y^{-1} \rp + \tr(\Gamma_Y)\rp
\nonumber \\
& = & 
\frac{1}{2} \lp \alpha \tr\lp\lp \sum_{i=1}^{l}
 F_{i}^T F_{i} \rp\Gamma_Y^{-1} \rp + \tr(\Gamma_Y)\rp
\nonumber \\
 \bar{L}_{l,X} & \triangleq  &  
\frac{1}{2} \lp \tr\lp\lp \sum_{i=1}^{l}
\lp F_{i,X} F_{i,X}^T \rp^2\rp\Gamma_X^{-1} \rp + \tr(\Gamma_X)\rp 
\nonumber \\
& = & 
\frac{1}{2} \lp \tr\lp\lp \sum_{i=1}^{l}
 F_{i} F_{i}^T \rp\Gamma_X^{-1} \rp + \tr(\Gamma_X)\rp ,
 \end{eqnarray}
 where we utilized
\begin{eqnarray}
   \label{eq:hrd9b0a0}
\lp F_{i,Y} F_{i,Y}^T \rp^2 = F_{i,Y} F_{i,Y}^TF_{i,Y} F_{i,Y}^T =
U_i\sqrt{S_i}\sqrt{S_i}U_i^T U_i\sqrt{S_i}\sqrt{S_i}U_i^T 
=U_iS_i^2U_i^T =U_iS_iV_i^TV_iS_iU_i^T =F_i^TF_i
 \end{eqnarray}
and
\begin{eqnarray}
   \label{eq:hrd9b0a1}
\lp F_{i,X} F_{i,X}^T \rp^2 = F_{i,X} F_{i,X}^TF_{i,X} F_{i,X}^T =
V_i\sqrt{S_i}\sqrt{S_i}V_i^T V_i\sqrt{S_i}\sqrt{S_i}V_i^T 
=V_iS_i^2V_i^T =V_iS_iU_i^TU_iS_iV_i^T =F_iF_i^T.
 \end{eqnarray}

Combining  (\ref{eq:kronrep5}), (\ref{eq:thm1eq2}), and  (\ref{eq:hrd9}), we arrive at 
\begin{eqnarray}\label{eq:hrd9a0}
 \lim_{n\rightarrow \infty}\mE\s_n(H) & =  & \max_{\tr(R_X^TR_X)=R_Y^TR_Y)=1} \lp    \lim_{n\rightarrow \infty} \frac{1}{\sqrt{n}} \mE \xi_a \rp  
\nonumber \\
& \geq & 
\frac{1}{2} \max_{\tr(R_X^TR_X)=\tr(R_Y^TR_Y)=1} 
 \Bigg ( \Bigg . 
   \min_{\Gamma_Y=\Gamma_Y^T} \lp \alpha \tr\lp\lp \sum_{i=1}^{l}
F_{i}^TF_{i}\rp \Gamma_Y^{-1} \rp + \tr(\Gamma_Y) \rp 
\nonumber \\
& & 
+ 
  \min_{\Gamma_X=\Gamma_X^T} \lp \tr\lp\lp \sum_{i=1}^{l}
F_{i}F_{i}^T \rp \Gamma_X^{-1} \rp + \tr(\Gamma_X) \rp 
\Bigg . \Bigg ) .
\end{eqnarray}
Solving the inner minimizations gives optimal $\Gamma_Y$ and $\Gamma_X$
\begin{eqnarray}\label{eq:hrd9a0a0}
\tilde{\Gamma}_Y =  \sqrt{\alpha\sum_{i=1}^{l} F_{i}^TF_{i}} \quad \mbox{and}\quad
\tilde{\Gamma}_X =  \sqrt{\sum_{i=1}^{l} F_{i}F_{i}^T} , 
\end{eqnarray} 
and
\begin{eqnarray}\label{eq:hrd9a1}
 \lim_{n\rightarrow \infty}\mE\s_n(H)  
& \geq & 
  \max_{\tr(R_X^TR_X)=\tr(R_Y^TR_Y)=1} \lp   \sqrt{\alpha}  \tr \sqrt{ \sum_{i=1}^{l}
F_i^TF_i } +   \tr \sqrt{ \sum_{i=1}^{l}
F_iF_i^T } \rp  
\nonumber \\
& \geq & 
  \max_{\substack{\tr(R_X^TR_X)=1\\ \tr(R_Y^TR_Y)=1}} \Bigg ( \Bigg .
   \sqrt{\alpha} \tr \sqrt{ \sum_{i=1}^{l}
( R_XA_i^TR_Y^T )^T( R_XA_i^TR_Y^T ) } 
 +  \tr \sqrt{ \sum_{i=1}^{l}
( R_XA_i^TR_Y^T )( R_XA_i^TR_Y^T )^T } \Bigg . \Bigg )  
\nonumber \\
&  = & 
  \max_{\substack{\tr(R_X^TR_X)=1\\ \tr(R_Y^TR_Y)=1}}\Bigg ( \Bigg . 
  \sqrt{\alpha}  \tr\sqrt{ \sum_{i=1}^{l}
A_iR_X^TR_X  A_i^TR_Y^TR_Y }
 +
  \tr\sqrt{ \sum_{i=1}^{l}
 A_i^TR_Y^TR_Y A_iR_X^TR_X } \Bigg .\Bigg )
 \nonumber \\
&  = & 
 \max_{\substack{S_Y=S_Y^T\succeq 0,S_X=S_X^T\succeq 0 \\
  \tr(S_Y)= \tr(S_X)=1} }\lp  \sqrt{\alpha} \tr\sqrt{ \sum_{i=1}^{l}
 (A_i S_X)(A_i^T S_Y) }  +  \tr\sqrt{ \sum_{i=1}^{l}
 (A_i^T S_Y)(A_i S_X) }  \rp .
\end{eqnarray}
This is effectively an asymmetric analogue to (38) in \cite{Stojnicleh26} (see also, Proposition 3.3 in \cite{CollYama26}, Theorem 1.9 in \cite{Kunisky26}, and \cite{bandfree23,bandfree24} for related results).   

The above allows to formulate the following corollary of Theorem \ref{thm:thm1a0}.
\begin{corollary}
\label{cor:cor1a0} Assume the setup of Theorem \ref{thm:thm1a0} (effectively the setup of Theorem \ref{thm:thm1} with $A_0=0$ and $\alpha=\lim_{n\rightarrow \infty} \frac{m}{n}$ not necessarily equal to $1$). Then 
\begin{eqnarray}\label{eq:cor1eq1}
 \lim_{n\rightarrow \infty}\mE\s_n(H)  
  &  \geq & 
 \max_{\substack{S_Y=S_Y^T\succeq 0,S_X=S_X^T\succeq 0 \\
  \tr(S_Y)= \tr(S_X)=1} }\lp  \sqrt{\alpha} \tr\sqrt{ \sum_{i=1}^{l}
 (A_i S_X)(A_i^T S_Y) }  +  \tr\sqrt{ \sum_{i=1}^{l}
 (A_i^T S_Y)(A_i S_X) }  \rp .
\end{eqnarray}
\end{corollary}

\begin{proof}
  Follows from Theorem \ref{thm:thm1a0} and (\ref{eq:hrd9})-(\ref{eq:hrd9a1}).
\end{proof}

\vspace{.15in}

\noindent \underline{\textbf{\emph{Practical implementation}}} 
\vspace{.05in}

Corollary \ref{cor:cor1a0} provides an elegant characterization of $ \lim_{n\rightarrow \infty}\mE\s_n(H)$. It relies on a single optimization, and for small $k_1$ and $k_2$ (which govern the sizes of $S_X$ and $S_Y$), it can be solved numerically. To create a bit more general numerical framework, it may actually be more convenient to work with (\ref{eq:hrd9a0}). We show that next.

First, we recall (\ref{eq:inteq1ab0b0})
\begin{eqnarray}
   \label{eq:numhrd6a15a0}
 F_i^TF_i= R_YA_iR_X^T R_XA_i^TR_Y^T \quad\mbox{and}\quad
 F_iF_i^T= R_XA_i^TR_Y^TR_YA_iR_X^T . 
\end{eqnarray}
Plugging these expressions into  (\ref{eq:hrd9a0}), gives
\begin{eqnarray}\label{eq:numhrd9a0}
 \lim_{n\rightarrow \infty}\mE\s_n(H)  
& \geq & 
\frac{1}{2} \max_{\tr(R_X^TR_X)=\tr(R_Y^TR_Y)=1} 
 \Bigg ( \Bigg . 
   \min_{\Gamma_Y=\Gamma_Y^T} \lp \alpha \tr\lp\lp \sum_{i=1}^{l}
R_YA_iR_X^T R_XA_i^TR_Y^T \rp \Gamma_Y^{-1} \rp + \tr(\Gamma_Y) \rp 
\nonumber \\
& & 
+ 
  \min_{\Gamma_X=\Gamma_X^T} \lp \tr\lp\lp \sum_{i=1}^{l}
R_XA_i^TR_Y^TR_YA_iR_X^T \rp \Gamma_X^{-1} \rp + \tr(\Gamma_X) \rp 
\Bigg . \Bigg ) .
\end{eqnarray}
Taking $\Gamma_X\rightarrow R_XZ_XR_X^T$ and $\Gamma_Y\rightarrow R_YZ_YR_Y^T$ (with $Z_X=Z_X^T\succeq 0$ and $Z_Y=Z_Y^T\succeq 0$; as (\ref{eq:hrd9a0a0}) shows, $\Gamma_X\succeq 0$ and  $\Gamma_Y\succeq 0$), we further find 
\begin{eqnarray}\label{eq:numhrd9a1}
 \lim_{n\rightarrow \infty}\mE\s_n(H)  
& \geq & 
\frac{1}{2} \max_{\substack{\tr(R_X^TR_X)=1\\\tr(R_Y^TR_Y)=1}} 
 \Bigg ( \Bigg . 
    \min_{Z_Y=Z_Y^T\succeq 0} \lp \alpha \tr\lp\lp \sum_{i=1}^{l}
A_iR_X^T R_XA_i^T\rp Z_Y^{-1} \rp + \tr(R_YZ_YR_Y^T) \rp 
\nonumber \\
& & 
+ 
  \min_{Z_X=Z_X^T\succeq 0} \lp \tr\lp\lp \sum_{i=1}^{l}
A_i^TR_Y^TR_YA_i \rp Z_X^{-1} \rp + \tr(R_XZ_XR_X^T) \rp 
\Bigg . \Bigg )
\nonumber \\
& = & 
\frac{1}{2} \max_{\substack{\tr(R_X^TR_X)=1\\\tr(R_Y^TR_Y)=1}} 
\min_{\substack{Z_Y=Z_Y^T\succeq 0\\Z_X=Z_X^T\succeq 0}} \Bigg ( \Bigg . 
    \alpha \tr\lp\lp \sum_{i=1}^{l}
A_iR_X^T R_XA_i^T\rp Z_Y^{-1} \rp + \tr(R_YZ_YR_Y^T)  
\nonumber \\
& & 
+ 
   \tr\lp\lp \sum_{i=1}^{l}
A_i^TR_Y^TR_YA_i \rp Z_X^{-1} \rp + \tr(R_XZ_XR_X^T)  
\Bigg . \Bigg )
\nonumber \\
& = & 
\frac{1}{2} 
\min_{\substack{Z_Y=Z_Y^T\succeq 0\\Z_X=Z_X^T\succeq 0}} 
\max_{\substack{\tr(R_X^TR_X)=1\\\tr(R_Y^TR_Y)=1}} 
\Bigg ( \Bigg . 
    \alpha \tr\lp\lp \sum_{i=1}^{l}
A_iR_X^T R_XA_i^T\rp Z_Y^{-1} \rp + \tr(R_YZ_YR_Y^T)  
\nonumber \\
& & 
+ 
   \tr\lp\lp \sum_{i=1}^{l}
A_i^TR_Y^TR_YA_i \rp Z_X^{-1} \rp + \tr(R_XZ_XR_X^T)  
\Bigg . \Bigg )
\nonumber \\
& = & 
\frac{1}{2} 
\min_{\substack{Z_Y=Z_Y^T\succeq 0\\Z_X=Z_X^T\succeq 0}} 
\max_{\substack{\tr(R_X^TR_X)=1\\\tr(R_Y^TR_Y)=1}} 
\Bigg ( \Bigg . 
    \alpha \tr\lp R_X \lp \sum_{i=1}^{l}
 A_i^T Z_Y^{-1} A_i\rp R_X^T \rp + \tr(R_YZ_YR_Y^T)  
\nonumber \\
& & 
+ 
   \tr\lp R_Y \lp \sum_{i=1}^{l}
  A_i  Z_X^{-1} A_i^T\rp R_Y^T \rp + \tr(R_XZ_XR_X^T)  
\Bigg . \Bigg )
\nonumber \\
& = & 
\frac{1}{2} 
\min_{\substack{Z_Y=Z_Y^T\succeq 0\\Z_X=Z_X^T\succeq 0}} 
\max_{\substack{\tr(R_X^TR_X)=1\\\tr(R_Y^TR_Y)=1}} 
\Bigg ( \Bigg . 
     \tr\lp R_X \lp Z_X +  \alpha \sum_{i=1}^{l}
 A_i^T Z_Y^{-1} A_i\rp R_X^T \rp    
\nonumber \\
& & 
+ 
   \tr\lp R_Y \lp Z_Y + \sum_{i=1}^{l}
  A_i  Z_X^{-1} A_i^T\rp R_Y^T \rp   
\Bigg . \Bigg )
\nonumber \\
& = & 
\frac{1}{2} 
\min_{\substack{Z_Y=Z_Y^T\succeq 0\\Z_X=Z_X^T\succeq 0}} 
 \Bigg ( \Bigg . 
     \lambda_n \lp Z_X +  \alpha \sum_{i=1}^{l}
 A_i^T Z_Y^{-1} A_i \rp    
 + 
   \lambda_n \lp Z_Y + \sum_{i=1}^{l}
  A_i  Z_X^{-1} A_i^T \rp   
\Bigg . \Bigg )
 ,
\end{eqnarray}
where the second equality follows since the objective is linear in $R_X^TR_X$ and $R_Y^TR_Y$ and quasi-convex (actually even convex) in $Z_X$ and $Z_Y$ on $Z_X\succeq 0$ and $Z_Y\succeq 0$. The above is a direct RDT non-square asymmetric analogue to symmetric Lehner formula (\ref{eq:amat1a0b2}). Given its elegance and clear resembling of (\ref{eq:amat1a0b2}), we formalize it in the following corollary.

\begin{corollary}
\label{cor:cor2a0} Assume the setup of Theorem \ref{thm:thm1a0} and Corollary \ref{cor:cor1a0} (effectively the setup of Theorem \ref{thm:thm1} with $A_0=0$ and $\alpha=\lim_{n\rightarrow \infty} \frac{m}{n}$ not necessarily equal to $1$). Let
 \begin{equation}
 \label{eq:cor2eq1} 
 \cA_0 \triangleq \{A_i\}_{i=0,1,\dots,l} \quad \mbox{with} \quad A_0=0. 
 \end{equation}
 Then 
\begin{equation}\label{eq:cor2eq2}
 \lim_{n\rightarrow \infty}\mE\s_n(H)  
    \geq 
\frac{1}{2} 
\min_{\substack{Z_Y=Z_Y^T\succeq 0\\Z_X=Z_X^T\succeq 0}} 
 \Bigg ( \Bigg . 
     \lambda_n \lp Z_X +  \alpha \sum_{i=1}^{l}
 A_i^T Z_Y^{-1} A_i \rp    
 + 
   \lambda_n \lp Z_Y + \sum_{i=1}^{l}
  A_i  Z_X^{-1} A_i^T \rp   
\Bigg . \Bigg )
 \triangleq \bar{\rho}_n(\cA_0)
 .
\end{equation}
\end{corollary}

\begin{proof}
  Follows from (\ref{eq:hrd9a0}) and (\ref{eq:numhrd6a15a0})-(\ref{eq:numhrd9a1}).
\end{proof}

Compared to the classical symmetric Lehner's formulas (\ref{eq:amat1a0b2}), (\ref{eq:cor2eq2})  uncovers presence of  a decoupling phenomenon. Namely, in the obtained characterization one clearly distinguishes contributions of the two leading eigenvalues associated to $X$ and $Y$ related deterministic parts of the system. However, decoupling is interlaced as the optimizing variables $Z_X$ and $Z_Y$  contribute to both parts.

In addition to providing an asymmetric analogue to Lehner's formula, Corollary \ref{cor:cor2a0} also provides a convenient analytical formulation amenable to simple numerical implementations. Namely, one first observes
\begin{align}\label{eq:numprfeq0a0}
&  &   \min_{\substack{Z_Y=Z_Y^T\succeq 0\\Z_X=Z_X^T\succeq 0}} 
 &  \hspace{.3in}  \Bigg ( \Bigg . 
     \lambda_n \lp Z_X +  \alpha \sum_{i=1}^{l}
 A_i^T Z_Y^{-1} A_i \rp    
 + 
   \lambda_n \lp Z_Y + \sum_{i=1}^{l}
  A_i  Z_X^{-1} A_i^T \rp   
\Bigg . \Bigg )   
\nonumber  \\
  \Longleftrightarrow  &  & \min_{
   \tilde{\lambda}_1,\tilde{\lambda}_2,Z_Y=Z_Y^T\succeq 0,Z_X=Z_X^T\succeq 0} 
   &  \hspace{.3in} \tilde{\lambda}_1+\tilde{\lambda}_2
\nonumber \\
  & &  \mbox{subject to } &     \hspace{.3in}   \lambda_n \lp Z_X +  \alpha \sum_{i=1}^{l}
 A_i^T Z_Y^{-1} A_i \rp    
 \leq \tilde{\lambda}_1      
\nonumber \\
  & &   &   \hspace{.3in}     \lambda_n \lp Z_Y +  \sum_{i=1}^{l}
 A_i Z_X^{-1} A_i^T \rp    
 \leq \tilde{\lambda}_2  
 .    
\end{align}
Utilizing the standard Schur complement matrix definiteness characterization, the following SDP (semi-definite programming) can be formulated (for a symmetric square analogue, see, e.g., \cite{Kunisky26}; for a different asymmetric non-square alternative, see also the last part of Section 
 \ref{sec:nqhandlerdA0})
\begin{eqnarray}
\label{eq:numprfeq2}
\bar{\rho}_n(\cA_0) = \frac{1}{2}\min_{\tilde{\lambda}_1,\tilde{\lambda}_2,Z_X,Z_Y}  & & \tilde{\lambda}_1+\tilde{\lambda}_2
\nonumber \\
\mbox{subject to } & & 
\begin{bmatrix}
\tilde{\lambda}_1 I_{k_2\times k_2} - \alpha \sum_{i=1}^{l}
 A_i^T Z_Y A_i    & I_{k_2\times k_2} \\ I_{k_2\times k_2}  &  Z_X 
\end{bmatrix} \succeq 0
\nonumber \\
& & 
\begin{bmatrix}
\tilde{\lambda}_2 I_{k_1\times k_1} - \sum_{i=1}^{l}
 A_i Z_X A_i^T    & I_{k_1\times k_1} \\ I_{k_1\times k_1}  &  Z_Y 
\end{bmatrix} \succeq 0
     \nonumber \\
 & &  Z_X=Z_X^T\succeq 0, Z_Y=Z_Y^T \succeq 0.
  \end{eqnarray}
To see why (\ref{eq:numprfeq2}) corresponds to the left-hand side of (\ref{eq:cor2eq2}), one notes
that the Schur complement positive semi-definiteness condition gives 
\begin{eqnarray}
\label{eq:numprfeq3}
 \begin{bmatrix}
\tilde{\lambda}_1 I_{k_2\times k_2} - \alpha \sum_{i=1}^{l}
 A_i^T Z_Y A_i    & I_{k_2\times k_2} \\ I_{k_2\times k_2}  &  Z_X 
\end{bmatrix} \succeq 0
& \Longleftrightarrow & 
\tilde{\lambda}_1 I_{k_2\times k_2} - \alpha \sum_{i=1}^{l}
 A_i^T Z_Y A_i   - I_{k_2\times k_2}  Z_X^{-1} I_{k_2\times k_2}, Z_X\succeq 0 
\nonumber \\
\begin{bmatrix}
\tilde{\lambda}_2 I_{k_1\times k_1} - \sum_{i=1}^{l}
 A_i Z_X A_i^T    & I_{k_1\times k_1} \\ I_{k_1\times k_1}  &  Z_Y 
\end{bmatrix} \succeq 0
& \Longleftrightarrow & 
\tilde{\lambda}_1 I_{k_2\times k_2} -  \sum_{i=1}^{l}
 A_i Z_X A_i^T   - I_{k_1\times k_1}  Z_Y^{-1} I_{k_1\times k_1}, Z_Y\succeq 0  
 .
 \nonumber \\
  \end{eqnarray}
A simple change of variables $Z_Y\rightarrow Z_Y^{-1}$ and $Z_X\rightarrow Z_X^{-1}$ then gives that the conditions on the left-hand sides of  (\ref{eq:numprfeq3}) precisely match the constraints in 
(\ref{eq:numprfeq0a0}). (To avoid overwhelming presentation with many tiny details, we always assume that the choice of $A_i$ is such that using $\inf$ and $\min$ in (\ref{eq:numprfeq2}) and all other optimization programs similar to it, brings not much of practically relevant difference.)

To see how all of above practically works, we take  $k_1=2$,  $k_2=3$,  $l=3$ and a particular choice  of $A_i$
 \begin{eqnarray}
 \label{eq:numsimeq1}
A_0 = 0,  \quad 
A_1 = \begin{bmatrix}
 2 & -1 & .3 \\ -2  & 3  &-2 
    \end{bmatrix},  \quad 
A_2 = \begin{bmatrix}
1 & -1 & 1 \\ .3  &  -1  & -.7 
    \end{bmatrix},  \quad 
A_3 = \begin{bmatrix}
1.4 & -2 & 1 \\ .5 & -1 & -2
    \end{bmatrix}. 
   \end{eqnarray}
Solving numerically the program given in (\ref{eq:numprfeq2})  for $\alpha=1$ and different values of $n$  produces results shown in Table \ref{tab:tab00}. Already for fairly small $n$ (on the order of a few hundreds), simulated results very closely approach the theoretical predictions (obtained in an infinite dimensional context).

\begin{table}[h]
\caption{$\mE(\s_n(H))$ and $\mbox{Std}(\s_n(H))$ as a function of $n$; $A_0=0$, $\alpha=1$; \textbf{theory}/\bl{\textbf{simulated}}}\vspace{.1in}
\centering
\def\arraystretch{1.2}
\begin{tabular}{||c||c|c|c|c|c||}\hline\hline
 \hspace{-0in}$n$                                              & $100$    &  $200$ &  $500$     & $1000$  & $\infty$ \\ \hline\hline
$\mE(\s_n(H))$                                       & $ \bl{\mathbf{10.068 }}   $  & $   \bl{\mathbf{ 10.132 }}  $  & $   \bl{\mathbf{ 10.185 }} $  
& $ \bl{\mathbf{ 10.208 }} $  & $ \mathbf{10.243} $ \\ \hline
$\mbox{Std}(\s_n(H))$                                       & $ \bl{\mathbf{.165 }} $  & $   \bl{\mathbf{  .105  }} $  & $ \bl{\mathbf{   .058 }} $  
& $ \bl{\mathbf{ .033  }} $  & $ \mathbf{0} $ \\ \hline\hline
  \end{tabular}
\label{tab:tab00}
\end{table}

To showcase the applicability of the above machinery to both, square and non-square scenarios, we also solved 
(\ref{eq:numprfeq2}) for $\alpha=1.9$ (this corresponds to non-square Gaussians; the above choice of $A_i$ ensures their non-squareness as well). The obtained results are shown in Table \ref{tab:tab01} and again very accurately emulate the theoretical predictions for fairly small $n$.

\begin{table}[h]
\caption{$\mE(\s_n(H))$ and $\mbox{Std}(\s_n(H))$ as a function of $n$; $A_0=0$, $\alpha=1.9$; \textbf{theory}/\bl{\textbf{simulated}}}\vspace{.1in}
\centering
\def\arraystretch{1.2}
\begin{tabular}{||c||c|c|c|c|c||}\hline\hline
 \hspace{-0in}$n$                                              & $100$    &  $200$ &  $500$     & $1000$  & $\infty$ \\ \hline\hline
$\mE(\s_n(H))$                                       & $ \bl{\mathbf{12.105}}   $  & $   \bl{\mathbf{12.170 }}  $  & $   \bl{\mathbf{12.220}} $  
& $ \bl{\mathbf{12.239}} $  & $ \mathbf{12.275} $ \\ \hline
$\mbox{Std}(\s_n(H))$                                       & $ \bl{\mathbf{.154}} $  & $   \bl{\mathbf{ .098}} $  & $ \bl{\mathbf{ .052 }} $  
& $ \bl{\mathbf{ .032 }} $  & $ \mathbf{0} $ \\ \hline\hline
  \end{tabular}
\label{tab:tab01}
\end{table}

\subsubsubsection{Deformed centered scenario }
\label{sec:defnqhandlerdA0}

Certain practical scenarios are closely connected to the $A_0=0$ context discussed in the previous section. Estimating the error operator norm in the so-called sample covariance problems typically amounts to determining the largest eigenvalue of a slightly deformed matrix $H^TH$. In the Kronecker-Gaussian context of our interest here, for $A_D=A_D^T\in\mR^{k_2\times k_2}$, a precise mathematical analogue of such a deformation would be  
\begin{equation}\label{eq:defeq1}
\lim_{n\rightarrow \infty}\mE\lambda_n( A_D \otimes I + H^TH).
  \end{equation}
Recalling (\ref{eq:kronrep1}), we write
\begin{eqnarray}\label{eq:defkronrep1}
\lambda_n(A_D\otimes I + H^TH)  
 & = &
 \max_{\x\in\mS^{k_2  n}}  \x^T ( A_D\otimes I +  H^TH ) \x 
\nonumber \\
& = & 
 \max_{\x\in\mS^{k_2 n}}  
 \lp
 \x^T ( A_D\otimes I ) \x
 +
 \|H\x \|^2_2
 \rp
\nonumber \\
& = &
 \max_{\x\in\mS^{k_2 n}} 
 \lp 
  \x^T ( A_D\otimes I ) \x
  + \lp \max_{\y\in\mS^{k_1 m}}  \y^T H\x \rp^2
  \rp 
\nonumber \\
& = & 
 \max_{\x\in\mS^{k_2 n}} 
 \lp 
  \x^T ( A_D\otimes I ) \x
  + \lp
 \max_{\y\in\mS^{k_1 m}} \y^T\lp \frac{1}{\sqrt{n}} \sum_{i=1}^{l} A_i\otimes G_i \rp \x \rp^2\rp.
 \end{eqnarray}
Repeating (\ref{eq:kronrep2})-(\ref{eq:kronrep4}), we arrive at
\begin{eqnarray}\label{eq:defkronrep4}
\lambda_n(A_D\otimes I + H^TH )  
   =  
    \max_{\tr(\bar{X}^T\bar{X})=1} 
    \lp
    \tr(A_D \bar{X}^T\bar{X} )
    +
    \lp
   \max_{\tr(\bar{Y}^T\bar{Y})=1} \lp  \frac{1}{\sqrt{n}} \sum_{i=1}^{l} \tr(A_i^T \bar{Y}^T G_i \bar{X} ) \rp \rp^2\rp .
   \end{eqnarray}
Reutilizing (\ref{eq:inteq1ab0})-(\ref{eq:kronrep6}), we also find
\begin{eqnarray}\label{eq:defkronrep5}
\lambda_n(A_D\otimes I + H^TH)  
  & =  & \max_{\tr(R_X^TR_X)=\tr(R_Y^TR_Y)=1} \lp  \frac{1}{n}\xi_{a,D} \rp  ,
\end{eqnarray}
where   
\begin{eqnarray}\label{eq:defkronrep6}
\varphi_D = n \tr(A_D R_X^T R_X)\quad \mbox{and} \quad  \xi_{a,D}  & \triangleq &  \max_{X^TX=Y^TY=I} \lp \varphi_D +  \lp \sum_{i=1}^{l}  \tr(Y^T G_i X F_i ) \rp^2 \rp .
\end{eqnarray}
One can then further reutilize all the machinery prior to (\ref{eq:numhrd9a1}) to arrive to its following analogue
\begin{eqnarray}\label{eq:defkronrep7}
 \lim_{n\rightarrow \infty}\mE 
\lambda_n(A_D\otimes I + H^TH)   
& \geq  & 
\min_{\substack{Z_Y=Z_Y^T\succeq 0\\Z_X=Z_X^T\succeq 0}} 
\max_{\substack{\tr(R_X^TR_X)=1\\\tr(R_Y^TR_Y)=1}} 
\Bigg ( \Bigg . 
\tr(A_DR_X^TR_X) 
\nonumber \\
& & +
\Bigg ( \Bigg . 
\frac{1}{2} 
     \tr\lp R_X \lp Z_X +  \alpha \sum_{i=1}^{l}
 A_i^T Z_Y^{-1} A_i\rp R_X^T \rp    
\nonumber \\
& & 
+ 
\frac{1}{2} 
   \tr\lp R_Y \lp Z_Y + \sum_{i=1}^{l}
  A_i  Z_X^{-1} A_i^T\rp R_Y^T \rp   
\Bigg . \Bigg )^2
\Bigg . \Bigg )
  .
\end{eqnarray}
Relying further on the RDT methodologies, we constrain $\tr(A_DR_X^TR_X)=c_D $ and obtain
\begin{align}\label{eq:defkronrep8}
 & \lim_{n\rightarrow \infty}\mE 
\lambda_n(A_D\otimes I + H^TH)   
 \geq   
\nonumber \\
& \geq
\min_{\substack{Z_Y=Z_Y^T\succeq 0\\Z_X=Z_X^T\succeq 0}} 
\max_{\substack{\tr(R_X^TR_X)=1\\\tr(R_Y^TR_Y)=1\\ c_D}}
 \min_{\gamma_D}
\Bigg ( \Bigg . 
c_D
  +
\Bigg ( \Bigg .
\gamma_D \tr(A_DR_X^TR_X) -\gamma_D c_D
+ 
\frac{1}{2} 
     \tr\lp R_X \lp Z_X +  \alpha \sum_{i=1}^{l}
 A_i^T Z_Y^{-1} A_i\rp R_X^T \rp    
\nonumber \\
&  \hspace{.15in}
+ 
\frac{1}{2} 
   \tr\lp R_Y \lp Z_Y + \sum_{i=1}^{l}
  A_i  Z_X^{-1} A_i^T\rp R_Y^T \rp   
\Bigg . \Bigg )^2
\Bigg . \Bigg )
=
\max_{c_D} 
 \Bigg ( \Bigg . 
c_D
   +
 \zeta(c_D)^2
\Bigg . \Bigg )
  ,
\end{align}
where
\begin{equation}\label{eq:defkronrep9}
 \zeta(c_D)  =   
 \min_{\substack{Z_Y=Z_Y^T\succeq 0\\Z_X=Z_X^T\succeq 0\\\gamma_D}} 
 \Bigg ( \Bigg .
\lambda_n \lp \gamma_D A_D
+ 
\frac{1}{2} 
      \lp Z_X +  \alpha \sum_{i=1}^{l}
 A_i^T Z_Y^{-1} A_i \rp\rp    
+ 
\frac{1}{2} 
   \lambda_n \lp Z_Y + \sum_{i=1}^{l}
  A_i  Z_X^{-1} A_i^T\rp      
 -\gamma_D c_D
\Bigg . \Bigg )
   .
\end{equation}

The above results are summarized in the following theorem.
                
\begin{theorem}
\label{thm:defthm1} Assume  the setup of Theorem \ref{thm:thm1} with $A_0=0$ and $\alpha=\lim_{n\rightarrow \infty} \frac{m}{n}$ not necessarily equal to $1$. Additionally, let $A_D=A_D^T$ and let $\zeta(c_D)$ be as in (\ref{eq:defkronrep9}). Then
\begin{equation}\label{eq:defthm1eq2}
\lim_{n\rightarrow \infty}\mE 
\sqrt{\lambda_n(A_D\otimes I + H^TH) }  
    \geq 
\sqrt{\max_{c_D} 
 \Bigg ( \Bigg . 
c_D
   +
 \zeta(c_D)^2
\Bigg . \Bigg )}
 .
\end{equation}
\end{theorem}

\begin{proof}
  Follows from the preceding discussion.
\end{proof}

\vspace{.15in}

\noindent \underline{\textbf{\emph{Practical implementation}}} 
\vspace{.05in}

The following SDP can be used to determine $\zeta(c_D)$
\begin{eqnarray}
\label{eq:defkronrep10}
\zeta(c_D) =  \min_{\tilde{\lambda}_1,\tilde{\lambda}_2,Z_X,Z_Y,\gamma_D}  & & \tilde{\lambda}_1+\frac{1}{2}\tilde{\lambda}_2 -\gamma_D c_D
\nonumber \\
\mbox{subject to } & & 
\begin{bmatrix}
\tilde{\lambda}_1 I_{k_2\times k_2} - \frac{1}{2}\alpha \sum_{i=1}^{l} 
 A_i^T Z_Y A_i -\gamma_D A_D   & I_{k_2\times k_2} \\ I_{k_2\times k_2}  &  2Z_X 
\end{bmatrix} \succeq 0
\nonumber \\
& & 
\begin{bmatrix}
\tilde{\lambda}_2 I_{k_1\times k_1} - \sum_{i=1}^{l}
 A_i Z_X A_i^T    & I_{k_1\times k_1} \\ I_{k_1\times k_1}  &  Z_Y 
\end{bmatrix} \succeq 0
     \nonumber \\
 & &  Z_X=Z_X^T\succeq 0, Z_Y=Z_Y^T \succeq 0.
  \end{eqnarray}
To  practically implement the above program, we take  $k_1=2$,  $k_2=3$,  $l=3$ and $A$ matrices   ($A_0=0$)
 \begin{equation}
 \label{eq:defnumsimeq1}
A_D = 30\times \begin{bmatrix}
 2.4  & 1 & -.3 \\ 1 & 2 &  .5 \\ -.3 & .5 & 3  
    \end{bmatrix},  \quad 
A_1 = \begin{bmatrix}
 2 & -1 & .3 \\ -2  & 3  &-2 
    \end{bmatrix},  \quad 
A_2 = \begin{bmatrix}
1 & -1 & 1 \\ .3  &  -1  & -.7 
    \end{bmatrix},  \quad 
A_3 = \begin{bmatrix}
1.4 & -2 & 1 \\ .5 & -1 & -2
    \end{bmatrix}. 
   \end{equation}
Solving (\ref{eq:defkronrep10}) numerically for $\alpha=1.9$ and different values of $n$  produces results shown in Table \ref{tab:tab02}. As earlier, for fairly small $n$, simulated results very closely approach the theoretical predictions.

\begin{table}[h]
\caption{$\mE\sqrt{\lambda_n(A_D\otimes I + H^TH)}$ and $\mbox{Std}(\sqrt{\lambda_n(A_D\otimes I + H^TH)})$ as a function of $n$; $A_i$ as in (\ref{eq:defnumsimeq1}), $\alpha=1.9$; \textbf{theory}/\bl{\textbf{simulated}}}\vspace{.1in}
\centering
\def\arraystretch{1.2}
\begin{tabular}{||c||c|c|c|c|c||}\hline\hline
 \hspace{-0in}$n$                                              & $100$    &  $200$ &  $500$     & $1000$  & $\infty$ \\ \hline\hline
$\mE\sqrt{\lambda_n(A_D\otimes I + H^TH)}$                                       & $ \bl{\mathbf{13.641     }}   $  & $   \bl{\mathbf{  13.698  }}  $  & $   \bl{\mathbf{ 13.743  }} $  
& $ \bl{\mathbf{ 13.757 }} $  & $ \mathbf{13.784} $ \\ \hline
$\mbox{Std}(\sqrt{\lambda_n(A_D\otimes I + H^TH)})$                                       & $ \bl{\mathbf{ .131  }} $  & $   \bl{\mathbf{  .084   }} $  & $ \bl{\mathbf{   .047   }} $  
& $ \bl{\mathbf{  .029 }} $  & $ \mathbf{0} $ \\ \hline\hline
$\mE (c_D)$                                       & $ \bl{\mathbf{46.521  }}   $  & $   \bl{\mathbf{  46.268  }}  $  & $   \bl{\mathbf{  46.137  }} $  
& $ \bl{\mathbf{  46.051 }} $  & $ \mathbf{45.950} $ \\ \hline
$\mbox{Std} (c_D) $                                       & $ \bl{\mathbf{ 1.531  }} $  & $   \bl{\mathbf{ 1.112   }} $  & $ \bl{\mathbf{   .623  }} $  
& $ \bl{\mathbf{   .461  }} $  & $ \mathbf{0} $ \\ \hline\hline
  \end{tabular}
\label{tab:tab02}
\end{table}

\subsubsubsection{Non-centered scenario ($A_0\neq 0$)}
\label{sec:nqhandlerdA0}

When $A_0\neq 0$, one in general has $\varphi \neq 0$ which implies that simplifications utilized in the previous subsection are a priori inapplicable. Nonetheless, we have the following analogue to Theorem \ref{thm:thm1a0}.

\begin{theorem}
\label{thm:nqthm1a0} Assume the setup of Theorem \ref{thm:thm1a0} with $A_0\neq 0$. One then has  
\begin{eqnarray}
   \label{eq:nqthm1eq2}
\lim_{n\rightarrow \infty}  \frac{1}{\sqrt{n}}  \mE \xi_a  
 \geq 
 \lim_{n\rightarrow \infty}  \frac{1}{\sqrt{n}}  \mE  
    \max_{\substack{Y\in\mR^{n\times k_1}, Y^TY =I\\ X\in\mR^{n\times k_2}, X^TX =I}} \lp \varphi + \sum_{i=1}^{l}  \tr\lp G_i^{(4)} Y F_{i,Y}F_{i,Y}^T \rp  + \sum_{i=1}^{l} \tr\lp G_i^{(5)} X F_{i,X}F_{i,X}^T \rp  \rp   .
\end{eqnarray}
 \end{theorem}

\begin{proof}
Follows automatically from Theorems \ref{thm:thm1} and \ref{thm:thm1a0}. 
\end{proof}

\begin{remark}
\label{rem:rem1} Two differences compared to Theorem \ref{thm:thm1a0} should be noted: (i) a priori decoupling of optimizations over $X$ and $Y$ is not applicable any longer; and (ii) $m\neq n$ would not make much sense.
\end{remark}

To analyze the above optimization, we start by setting
\begin{eqnarray}
   \label{eq:nqhrd1}
 \tilde{L} & \triangleq &  
    \max_{\substack{Y\in\mR^{n\times k_1}, Y^TY =I\\ X\in\mR^{n\times k_2}, X^TX =I}} \lp \varphi + \sum_{i=1}^{l}  \tr\lp G_i^{(4)} Y F_{i,Y}F_{i,Y}^T \rp  + \sum_{i=1}^{l} \tr\lp G_i^{(5)} X F_{i,X}F_{i,X}^T \rp  \rp  
       \nonumber \\
     &  = &
     -      \min_{\substack{Y\in\mR^{n\times k_1}, Y^TY =I\\ X\in\mR^{n\times k_2}, X^TX =I}} \lp -\varphi + \sum_{i=1}^{l}  \tr\lp G_i^{(4)} Y F_{i,Y}F_{i,Y}^T \rp  + \sum_{i=1}^{l} \tr\lp G_i^{(5)} X F_{i,X}F_{i,X}^T \rp  \rp .
\end{eqnarray}
Then for $\Gamma_X=\Gamma_X^T$ and $\Gamma_Y=\Gamma_Y^T$ we write the Lagrangian 
\begin{eqnarray}
   \label{eq:nqhrd2}
 \cL & = &  -\sqrt{n} \tr(A_0^T R_Y^T Y^TX  R_X)  + \sum_{i=1}^{l}  \tr\lp G_i^{(4)} Y F_{i,Y}F_{i,Y}^T \rp  + \sum_{i=1}^{l} \tr\lp G_i^{(5)} X F_{i,X}F_{i,X}^T \rp  
 \nonumber \\
& &  + \tr(\Gamma_X X^TX )  - \tr(\Gamma_X)  + \tr(\Gamma_Y Y^TY )  - \tr(\Gamma_Y)  .
\end{eqnarray}
Combining (\ref{eq:nqhrd1}) and  (\ref{eq:nqhrd2})  with the strong duality, we have
\begin{eqnarray}
   \label{eq:nqhrd3}
 \tilde{L}  = - \min_{\substack{X\in\mR^{n\times k_2}\\ Y\in\mR^{n\times k_1} }}
 \max_{\substack{\Gamma_X=\Gamma_X^T \\ \Gamma_Y=\Gamma_Y^T } } \cL
 = 
- \max_{\substack{\Gamma_X=\Gamma_X^T\succeq 0 \\ \Gamma_Y=\Gamma_Y^T\succeq 0} }
 \min_{\substack{X\in\mR^{n\times k_2}\\ Y\in\mR^{n\times k_1} }}
  \cL.
\end{eqnarray}
 Let
\begin{eqnarray}
   \label{eq:nqhrd3a0}
\cZ_1^T & = &  \sum_{i=1}^{l} F_{i,Y}F_{i,Y}^TG_i^{(4)} 
\nonumber \\
\cZ_2^T & = &  \sum_{i=1}^{l} F_{i,X}F_{i,X}^TG_i^{(5)} 
\nonumber \\
\cZ_3 & = &  -\sqrt{n} R_XA_0^TR_Y^T.
 \end{eqnarray}
Then  (\ref{eq:nqhrd2}) can be rewritten as
\begin{eqnarray}
   \label{eq:nqhrd3a0a0}
 \cL  =   \tr(\cZ_3Y^TX )  + \tr\lp \cZ_1^T Y \rp  + \tr\lp \cZ_2^T X \rp
  + \tr(\Gamma_X X^TX )  - \tr(\Gamma_X)  + \tr(\Gamma_Y Y^TY )  - \tr(\Gamma_Y)  .
\end{eqnarray}
After taking  $Y$ and $X$ derivatives, we obtain
\begin{eqnarray}
   \label{eq:nql2hrd4}
 \frac{d\cL}{dY } & =  & \cZ_{1} + 2Y \Gamma_Y   
 + X\cZ_3
 \nonumber \\
 \frac{d\cL}{dX } & =  & \cZ_{2} + 2X\Gamma_X   
 + Y\cZ_3^T
.
\end{eqnarray}
Equalling  derivatives in (\ref{eq:nql2hrd4}) to zero gives 
\begin{eqnarray}
   \label{eq:nql2hrd4a1}
\cZ_{1} + 2Y \Gamma_Y   
 + X\cZ_3 & = & 0
 \nonumber \\
\cZ_{2} + 2X\Gamma_X   
 + Y\cZ_3^T  & = & 0
,
\end{eqnarray}
and
\begin{eqnarray}
   \label{eq:nql2hrd4a1a0}
 Y & = & -\frac{1}{2}  \cZ_{1} \Gamma_Y^{-1}  
 -\frac{1}{2} X\cZ_3\Gamma_Y^{-1}
\nonumber \\
X & = & -\frac{1}{2}  \cZ_{2} \Gamma_X^{-1}  
 -\frac{1}{2} Y\cZ_3^T\Gamma_X^{-1}
.
\end{eqnarray}
We then find
\begin{eqnarray}
   \label{eq:nql2hrd4a1a0}
 Y \Gamma_Y
 & = &  -\frac{1}{2}  \cZ_{1} 
 -\frac{1}{2} 
 \lp 
 -\frac{1}{2}  \cZ_{2} \Gamma_X^{-1}  
 -\frac{1}{2} Y\cZ_3^T\Gamma_X^{-1}
 \rp \cZ_3   
\nonumber \\
& = &  -\frac{1}{2}  \cZ_{1}  
  +\frac{1}{4}  \cZ_{2} \Gamma_X^{-1}\cZ_3   
 +\frac{1}{4} Y\cZ_3^T\Gamma_X^{-1}
  \cZ_3     
,
\end{eqnarray}
and
\begin{eqnarray}
   \label{eq:nql2hrd4a1a1}
 Y 
 & = &  \lp -\frac{1}{2}  \cZ_{1}  
  +\frac{1}{4}  \cZ_{2} \Gamma_X^{-1}\cZ_3  \rp 
 \lp  \Gamma_Y -\frac{1}{4} \cZ_3^T\Gamma_X^{-1}
  \cZ_3 \rp^{-1}     
.
\end{eqnarray}
Analogously
\begin{eqnarray}
   \label{eq:nql2hrd4a1a2}
 X 
 & = &  \lp -\frac{1}{2}  \cZ_{2}  
  +\frac{1}{4}  \cZ_{1} \Gamma_Y^{-1}\cZ_3^T  \rp 
 \lp  \Gamma_X -\frac{1}{4} \cZ_3\Gamma_Y^{-1}
  \cZ_3^T \rp^{-1}     
.
\end{eqnarray}
For the above $X$ and  $Y$, we find from (\ref{eq:nql2hrd4a1}) 
\begin{eqnarray}
   \label{eq:nql2hrd4a2a3}
  \cZ_{1}^TY & = & -2\Gamma_Y Y^TY  
 - \cZ_3^T X^TY
 \nonumber \\
  \cZ_{2}^TX& = & -2\Gamma_X X^TX  
 - \cZ_3 Y^TX.
\end{eqnarray}
A combination of  (\ref{eq:nqhrd3a0a0}) and (\ref{eq:nql2hrd4a2a3}) then gives
\begin{eqnarray}
   \label{eq:nql2hrd4a2a4}
\min_{Y,X}  \cL 
& = &
  - \tr\lp \Gamma_Y Y^TY\rp 
  - \tr\lp \Gamma_X X^TX\rp - \tr( \Gamma_Y )- \tr( \Gamma_X )
  - \tr\lp  \cZ_3Y^TX \rp   
\nonumber \\
& = &
\frac{1}{2}\tr\lp\cZ_{1}^TY\rp + \frac{1}{2}\tr\lp \cZ_{2}^TX \rp- \tr( \Gamma_Y )- \tr( \Gamma_X )
\nonumber \\
& = &
\frac{1}{2}\tr\lp
\lp -\frac{1}{2} \cZ_{1}^T  \cZ_{1}  
  +\frac{1}{4}  \cZ_{1}^T \cZ_{2} \Gamma_X^{-1}\cZ_3  \rp 
 \lp  \Gamma_Y -\frac{1}{4} \cZ_3^T\Gamma_X^{-1}
  \cZ_3 \rp^{-1}     
\rp 
\nonumber \\
& & 
+ \frac{1}{2}\tr\lp 
\lp -\frac{1}{2}  \cZ_{2}^T\cZ_{2}  
  +\frac{1}{4}  \cZ_{2}^T\cZ_{1} \Gamma_Y^{-1}\cZ_3^T  \rp 
 \lp  \Gamma_X -\frac{1}{4} \cZ_3\Gamma_Y^{-1}
  \cZ_3^T \rp^{-1}     
 \rp - \tr( \Gamma_Y )- \tr( \Gamma_X ) .
 \end{eqnarray}
The law of large numbers, concentrations, and (\ref{eq:nqhrd3})-(\ref{eq:nql2hrd4a2a4}) then give in the considered dimensional scenarios
\begin{equation}
   \label{eq:nql2hrd4a2a4a0}
\lim_{n\rightarrow \infty} \mE\frac{\tilde{L}}{\sqrt{n}}   =  
 -\max_{\substack{\Gamma_X=\Gamma_X^T\succeq 0 \\ \Gamma_Y=\Gamma_Y^T \succeq 0 } }
 \lim_{n\rightarrow \infty} \mE\frac{ \min_{X,Y} 
  \cL}{\sqrt{n}} 
  = 
- \max_{\substack{\Gamma_X=\Gamma_X^T\succeq 0 \\ \Gamma_Y=\Gamma_Y^T\succeq 0} }
\lp   \lim_{n\rightarrow \infty} \mE\frac{\tilde{L}_1}{\sqrt{n}} 
   +
    \lim_{n\rightarrow \infty} \mE\frac{\tilde{L}_2}{\sqrt{n}} 
- \tr( \Gamma_Y )- \tr( \Gamma_X ) \rp
,
\end{equation}
where
\begin{eqnarray}
   \label{eq:nql2hrd4a2a5}
 \tilde{L}_1& = &
\frac{1}{2}\tr\lp
\lp -\frac{1}{2} \cZ_{1}^T  \cZ_{1}  
  +\frac{1}{4}  \cZ_{1}^T \cZ_{2} \Gamma_X^{-1}\cZ_3  \rp 
 \lp  \Gamma_Y -\frac{1}{4} \cZ_3^T\Gamma_X^{-1}
  \cZ_3 \rp^{-1}     
\rp 
\nonumber \\
\tilde{L}_2
 &  = & 
 \frac{1}{2}\tr\lp 
\lp -\frac{1}{2}  \cZ_{2}^T\cZ_{2}  
  +\frac{1}{4}  \cZ_{2}^T\cZ_{1} \Gamma_Y^{-1}\cZ_3^T  \rp 
 \lp  \Gamma_X -\frac{1}{4} \cZ_3\Gamma_Y^{-1}
  \cZ_3^T \rp^{-1}     
 \rp .
\end{eqnarray}
Keeping in mind (\ref{eq:nqhrd3a0}) and cosmetic scalings $\Gamma_X=\Gamma_X\sqrt{n}$ and $\Gamma_Y=\Gamma_Y\sqrt{n}$, we further find
\begin{eqnarray}
   \label{eq:nql2hrd4a2a6}
\lim_{n\rightarrow \infty} \mE\frac{\tilde{L}_1}{\sqrt{n}}
 & = &
\lim_{n\rightarrow \infty} \mE\frac{1}{\sqrt{n}}
\frac{1}{2}\tr\lp
\lp -\frac{1}{2} \cZ_{1}^T  \cZ_{1}  
  +\frac{1}{4}  \cZ_{1}^T \cZ_{2} \Gamma_X^{-1}\cZ_3  \rp 
 \lp  \Gamma_Y -\frac{1}{4n} \cZ_3^T\Gamma_X^{-1}
  \cZ_3 \rp^{-1}     
\rp 
\nonumber \\
 & = &
-\frac{1}{4} \lim_{n\rightarrow \infty} \mE\frac{1}{n}
 \tr\lp
  \cZ_{1}^T  \cZ_{1}   
 \lp  \Gamma_Y -\frac{1}{4n} \cZ_3^T\Gamma_X^{-1}
  \cZ_3 \rp^{-1}     
\rp 
\nonumber \\
 & = &
-\frac{1}{4} \lim_{n\rightarrow \infty} \mE\frac{1}{n}
 \tr\lp
  \sum_{i=1}^{l} F_{i,Y}F_{i,Y}^TG_i^{(4)}  \lp  \sum_{i=1}^{l} F_{i,Y}F_{i,Y}^TG_i^{(4)}  \rp^T   
 \lp  \Gamma_Y -\frac{1}{4n} \cZ_3^T\Gamma_X^{-1}
  \cZ_3 \rp^{-1}     
\rp 
\nonumber \\
 & = &
-\frac{1}{4}    
 \tr\lp
  \sum_{i=1}^{l} \lp F_{i,Y}F_{i,Y}^T \rp^2  
 \lp  \Gamma_Y -\frac{1}{4} R_YA_0R_X^T \Gamma_X^{-1}
  R_XA_0^TR_Y^T \rp^{-1}     
\rp 
.
 \end{eqnarray}
Taking again $\Gamma_X\rightarrow R_XZ_XR_X^T$ and $\Gamma_Y\rightarrow R_YZ_YR_Y^T$ with $Z_X=Z_X^T\succeq 0$ and $Z_Y=Z_Y^T\succeq 0$, and recalling 
\begin{eqnarray}
   \label{eq:nqhrd6a15a0}
\lp F_{i,Y}F_{i,Y}^T\rp^2=F_i^TF_i= R_YA_iR_X^T R_XA_i^TR_Y^T \quad\mbox{and}\quad
\lp F_{i,X}F_{i,X}^T\rp^2=F_iF_i^T= R_XA_i^TR_Y^TR_YA_iR_X^T , 
\end{eqnarray}
we obtain
\begin{eqnarray}
   \label{eq:nql2hrd4a2a7}
\lim_{n\rightarrow \infty} \mE\frac{\tilde{L}_1}{\sqrt{n}}
  & = &
-\frac{1}{4}  
 \tr\lp
  \sum_{i=1}^{l} R_YA_iR_X^T R_XA_i^TR_Y^T  
 \lp  R_YZ_YR_Y^T -\frac{1}{4} R_YA_0 Z_X^{-1}
  A_0^TR_Y^T \rp^{-1}     
\rp 
\nonumber \\
  & = &
-\frac{1}{4}    
 \tr\lp
  \sum_{i=1}^{l} A_iR_X^T R_XA_i^T   
 \lp   Z_Y -\frac{1}{4} A_0 Z_X^{-1}
  A_0^T \rp^{-1}     
\rp 
\nonumber \\
  & = &
-\frac{1}{4}   
 \tr\lp
  R_X \lp \sum_{i=1}^{l} A_i^T   
 \lp   Z_Y -\frac{1}{4} A_0 Z_X^{-1}
  A_0^T \rp^{-1}     
 A_i \rp R_X^T
\rp.
 \end{eqnarray}
Repeating the above for $\tilde{L}_2$, gives in a completely analogous manner
\begin{eqnarray}
   \label{eq:nql2hrd4a2a7}
\lim_{n\rightarrow \infty} \mE\frac{\tilde{L}_2}{\sqrt{n}}
   & = &
-\frac{1}{4}    
 \tr\lp
  R_Y \lp \sum_{i=1}^{l} A_i   
 \lp   Z_X -\frac{1}{4} A_0^T Z_Y^{-1}
  A_0 \rp^{-1}     
 A_i^T \rp R_Y^T
\rp.
 \end{eqnarray}
Combining (\ref{eq:nqhrd3}), (\ref{eq:nql2hrd4a2a6}), and  (\ref{eq:nql2hrd4a2a7}) with scalings $Z_X\rightarrow \frac{1}{2} Z_X$ and $Z_Y\rightarrow \frac{1}{2} Z_Y$, we arrive at
\begin{eqnarray}
   \label{eq:nqhrd4}
\lim_{n\rightarrow \infty} \mE\frac{\tilde{L}}{\sqrt{n}}   
& = &  
 \frac{1}{2} \min_{\substack{Z_X=Z_X^T\succeq 0\\ Z_Y=Z_Y^T\succeq 0} }
   \Bigg ( \Bigg.   
 \tr\lp
  R_X \lp \sum_{i=1}^{l} A_i^T   
 \lp   Z_Y -  A_0 Z_X^{-1}
  A_0^T \rp^{-1}     
 A_i \rp R_X^T
\rp
\nonumber \\
& & +   
 \tr\lp
  R_Y \lp \sum_{i=1}^{l} A_i   
 \lp   Z_X -  A_0^T Z_Y^{-1}
  A_0 \rp^{-1}     
 A_i^T \rp R_Y^T
\rp 
\nonumber \\
& & 
+\tr(R_Y Z_YR_Y^T)
+\tr(R_X Z_X R_X^T)
   \Bigg .\Bigg )
.
\end{eqnarray}
From  (\ref{eq:kronrep5}), (\ref{eq:nqthm1eq2}), and  (\ref{eq:nqhrd4}), we then have 
\begin{eqnarray}\label{eq:nqhrd9a0}
 \lim_{n\rightarrow \infty}\mE\s_n(H) & =  & \max_{\tr(R_X^TR_X)=R_Y^TR_Y)=1} \lp    \lim_{n\rightarrow \infty} \frac{1}{\sqrt{n}} \mE \xi_a \rp  
\nonumber \\
& \geq & 
\frac{1}{2} \max_{\tr(R_X^TR_X)=\tr(R_Y^TR_Y)=1} 
 \Bigg ( \Bigg . 
     \min_{\substack{Z_X=Z_X^T\succeq 0\\ Z_Y=Z_Y^T\succeq 0} }
   \Bigg ( \Bigg.   
 \tr\lp
  R_X \lp \sum_{i=1}^{l} A_i^T   
 \lp   Z_Y -  A_0 Z_X^{-1}
  A_0^T \rp^{-1}     
 A_i \rp R_X^T
\rp
\nonumber \\
& & +   
 \tr\lp
  R_Y \lp \sum_{i=1}^{l} A_i   
 \lp   Z_X -  A_0^T Z_Y^{-1}
  A_0 \rp^{-1}     
 A_i^T \rp R_Y^T
\rp 
+\tr(R_Y Z_YR_Y^T)
+\tr(R_X Z_X R_X^T)
   \Bigg .\Bigg )
   \Bigg . \Bigg ) 
   \nonumber \\
& = & 
\frac{1}{2}     \min_{\substack{Z_X=Z_X^T\succeq 0\\ Z_Y=Z_Y^T\succeq 0} }
 \Bigg ( \Bigg . 
  \lambda_n\lp
   Z_X + \sum_{i=1}^{l} A_i^T   
 \lp   Z_Y -  A_0 Z_X^{-1}
  A_0^T \rp^{-1}     
 A_i \rp
\nonumber \\
& & +   
 \lambda_n\lp
  Z_Y + \lp \sum_{i=1}^{l} A_i   
 \lp   Z_X -  A_0^T Z_Y^{-1}
  A_0 \rp^{-1}     
 A_i^T \rp 
\rp 
    \Bigg . \Bigg ) 
   .
\end{eqnarray}

The above is summarized in the following theorem.
\begin{theorem} Assume the setup of Theorem \ref{thm:thm1} (with $A_0$ allowed to be nonzero). Then
\label{thm:nqthm2}  
\begin{eqnarray}\label{eq:nqthm2eq1}
 \lim_{n\rightarrow \infty}\mE\s_n(H)& \geq  & 
\frac{1}{2}     \min_{\substack{Z_X=Z_X^T\succeq 0\\ Z_Y=Z_Y^T\succeq 0} }
 \Bigg ( \Bigg . 
  \lambda_n\lp
   Z_X + \sum_{i=1}^{l} A_i^T   
 \lp   Z_Y -  A_0 Z_X^{-1}
  A_0^T \rp^{-1}     
 A_i \rp
\nonumber \\
& & +   
 \lambda_n\lp
  Z_Y + \lp \sum_{i=1}^{l} A_i   
 \lp   Z_X -  A_0^T Z_Y^{-1}
  A_0 \rp^{-1}     
 A_i^T \rp 
\rp 
    \Bigg . \Bigg ) \triangleq \bar{\rho}_n(\cA)
   .
\end{eqnarray}
\end{theorem}

\begin{proof}
  Follows from the preceding discussion.
\end{proof}

\begin{remark}
  \label{rem:rem2}
As stated earlier, to make writing easier, we assume throughout the analysis that the choice of $A_i$ is such that distinguishing between having $\min$ or alternatively $\inf$ in (\ref{eq:nqthm2eq1}) (and all other optimization programs similar to it), is of no practical relevance.
 \end{remark} 

\begin{remark}
  \label{rem:rem3}
 For $A_0=0$, one deduces the square variant ($\alpha=1$) of Corollary \ref{cor:cor2a0}.
\end{remark}

\vspace{.15in}

\noindent \underline{\textbf{\emph{Practical implementation}}} 
\vspace{.05in}

Theorem \ref{thm:nqthm2} provides an analytical characterization of $\s_n(H)$. Moreover, one can actually practically estimate the left-hand side of (\ref{eq:nqthm2eq1}). Namely, the idea is that for any admissible $Z_X$ and $Z_Y$, we introduce new variables $B_X=B_X^T$ and $B_Y=B_Y^T$ such that 
\begin{eqnarray}
\label{eq:nqprfeq1}
 Z_Y -  A_0 Z_X^{-1}
  A_0^T \succeq B_Y \quad \mbox{and} \quad 
 Z_X -  A_0^T Z_Y^{-1}
  A_0 \succeq B_X.
  \end{eqnarray}
We then utilizes again the Schur complement positive semi-definiteness characterization to formulate the following SDP 
\begin{eqnarray}
\label{eq:nqprfeq2}
\bar{\rho}_n(\cA) = \min_{Z_X,Z_Y,B_X,B_Y}  & & \lambda_1+\lambda_2
\nonumber \\
\mbox{subject to } & & 
\begin{bmatrix}
\lambda_1 I_{k_1\times k_1} & \begin{bmatrix}
      A_1 & A_2 & \dots & A_l 
    \end{bmatrix} \\ \begin{bmatrix}
      A_1 & A_2 & \dots & A_l 
    \end{bmatrix}^T &  I_{l\times l}\otimes B_X 
\end{bmatrix} \succeq 0
  \nonumber \\
& & \begin{bmatrix}
\lambda_2  I_{k_2\times k_2} & \begin{bmatrix}
      A_1^T & A_2^T & \dots & A_l^T 
    \end{bmatrix} \\ \begin{bmatrix}
      A_1^T & A_2^T & \dots & A_l^T 
    \end{bmatrix}^T &  I_{l\times l}\otimes B_Y 
\end{bmatrix} \succeq 0
  \nonumber \\
& & 
\begin{bmatrix}
      Z_Y -B_Y & A_0 \\A_0^T & Z_X
    \end{bmatrix}^T \succeq 0
   \nonumber \\
 & & 
 \begin{bmatrix}
      Z_X -B_X & A_0^T \\A_0 & Z_Y
    \end{bmatrix}^T  \succeq 0
   \nonumber \\
 & &  Z_X=Z_X^T, Z_Y=Z_Y^T, B_X=B_X^T, B_Y=B_Y^T .
  \end{eqnarray}

For  $k_1=2$,  $k_2=3$,  $l=3$ and a particular choice  of $A_i$
 \begin{equation}
 \label{eq:simeq1}
A_0 = \begin{bmatrix}
.2 &  1 & -2.3 \\ -.2 & 2.8 & -1
    \end{bmatrix},  \quad 
A_1 = \begin{bmatrix}
 2 & -1 & .3 \\ -2  & 3  &-2 
    \end{bmatrix},  \quad 
A_2 = \begin{bmatrix}
1 & -1 & 1 \\ .3  &  -1  & -.7 
    \end{bmatrix},  \quad 
A_3 = \begin{bmatrix}
1.4 & -2 & 1 \\ .5 & -1 & -2
    \end{bmatrix}, 
   \end{equation}
the obtained results are shown in Table \ref{tab:tab1}. As can be seen, already for fairly small $n$ (on the order of a few hundreds), simulated values approach theoretical predictions.

\begin{table}[h]
\caption{$\mE(\s_n(H))$ and $\mbox{Std}(\s_n(H))$ as a function of $n$; \textbf{theory}/\bl{\textbf{simulated}}}\vspace{.1in}
\centering
\def\arraystretch{1.2}
\begin{tabular}{||c||c|c|c|c|c||}\hline\hline
 \hspace{-0in}$n$                                              & $100$    &  $200$ &  $500$     & $1000$  & $\infty$ \\ \hline\hline
$\mE(\s_n(H))$                                       & $ \bl{\mathbf{10.942}}   $  & $   \bl{\mathbf{11.020}}  $  & $   \bl{\mathbf{11.065}} $  
& $ \bl{\mathbf{11.078}} $  & $ \mathbf{11.119} $ \\ \hline
$\mbox{Std}(\s_n(H))$                                       & $ \bl{\mathbf{.150}} $  & $   \bl{\mathbf{ .094 }} $  & $ \bl{\mathbf{.048}} $  
& $ \bl{\mathbf{.033}} $  & $ \mathbf{0} $ \\ \hline\hline
  \end{tabular}
\label{tab:tab1}
\end{table}

It is worth noting that the above also applies to the $A_0=0$ scenario. In that case, the third and fourth constraints are redundant and can be removed, given that $B_X=Z_X$ and $B_Y=Z_Y$. This establishes a square alternative to (\ref{eq:numprfeq2}). To obtain the non-square alternative, it is sufficient to multiply the $A_i$ terms in the second constraint by $\sqrt{\alpha}$. Whether this is a preferable alternative depends on practical numerical implementations; while the SDP in (\ref{eq:nqprfeq2}) is larger, it possesses a very particular sparse structure.

In \cite{ParHand25}, the singular values of free operators are studied in general asymmetric scenarios. The main results, Theorem 1.2 and Corollary 1.3, provide convenient analytical characterizations. Interestingly, the optimizing mechanism in (\ref{eq:nqprfeq2}) allows automatic formulation of the SDP analogues for the key objects in Theorem 1.2 and Corollary 1.3 with only trivial adjustments. However, within the general intrinsic freeness context ($n=1$), care must be taken with these formulations as the size of the resulting SDPs ($\sim \max(k_1,k_2)$) is typically very large. This stands in stark contrast to asymptotic freeness, where $n\rightarrow\infty$ and $\max(k_1,k_2)$ need not be large.

\section{Conclusion}
\label{sec:conc}

In \cite{Lehner99}, Lehner obtained elegant analytical characterizations of spectral edges for a class of semicircular free operators. By strong asymptotic freeness \cite{HaagThor05,Schultz05}, these precisely match the spectral edges of Kronecker-Gaussian matrices. Circumventing classical random matrix theory spectral methods, we developed a powerful framework in \cite{Stojnicleh26} to study these objects based on Random Duality Theory (RDT). As a concrete application, we reproved a definite variant of Lehner's formula.

We have further explored the power of RDT by utilizing the framework from \cite{Stojnicleh26}, achieving significant progress in several directions. First, we provide matching lower bounds for Lehner's formula in the indefinite case, where the deterministic matrix coefficients are not necessarily positive semi-definite. Second, we study the asymmetric non-square variants of Kronecker-Gaussian matrices, formulate corresponding RDT analogues for Lehner's formula, and demonstrate that they lower-bound the spectral edges. Finally, we uncover a remarkable interlaced decoupling property within these asymmetric analogues.

The supporting algorithmic and numerical aspects are discussed as well, with three components being of particular relevance: (i)  We demonstrate that the obtained asymmetric Lehner's formula RDT analogues can be cast as semi-definite programs (SDPs). (ii)  We solve these programs numerically in several concrete cases. (iii)  We conduct numerical simulations and compare the results to theoretical predictions. Even though the theoretical analyses assume an infinite dimensional scenario, a striking similarity is observed between the simulations and theory for relatively small problem dimensions (on the order of a few hundreds).

\section*{Acknowledgment}

The author would like to thank Ramon van Handel for a fruitful discussion on several related topics and in particular for pointing out the relevance of the problems studied here.

%
%
%
%
%
%
%

\begin{singlespace}
\bibliographystyle{plain}
\bibliography{nflgscompyxRefs}

\begin{thebibliography}{10}

\bibitem{Adam10}
R.~Adamczak, A.~E. Litvak, A.~Pajor, and N.~Tomczak-Jaegermann.
\newblock Quantitative estimates of the convergence of the empirical covariance
  matrix in log-concave ensembles.
\newblock {\em Journal of the American Mathematical Society}, 23(2):535--561,
  2010.

\bibitem{Adamczak11}
R.~Adamczak, A.~E. Litvak, A.~Pajor, and N.~Tomczak-Jaegermann.
\newblock Sharp bounds on the rate of convergence of the empirical covariance
  matrix.
\newblock {\em Comptes Rendus Mathématique}, 349(3--4):195--200, 2011.

\bibitem{Anderson13}
G.~W. Anderson.
\newblock Convergence of the largest singular value of a polynomial in
  independent {W}igner matrices.
\newblock {\em Annals of Probability}, 41(3B):2103--2181, 2013.

\bibitem{BBP05}
J.~Baik, G.~Ben~Arous, and S.~Peche.
\newblock Phase transition of the largest eigenvalue for non-null complex
  sample covariance matrices.
\newblock {\em The Annals of Probability}, 33(5):1643--1697, 2005.

\bibitem{bandfree24}
A.~S. Bandeira, G.~Cipolloni, D.~Schroder, and R.~van Handel.
\newblock Matrix concentration inequalities and free probability ii.
  {T}wo-sided bounds and applications.
\newblock 2024.
\newblock available online at \bl{\url{http://arxiv.org/abs/2406.11453}}.

\bibitem{BandHand16}
A.~S. Bandeira and R.~van Handel.
\newblock Sharp nonasymptotic bounds on the norm of random matrices with
  independent entries.
\newblock {\em The Annals of Probability}, 44(4):2479--2506, 2016.

\bibitem{bandfree23}
A.S. Bandeira, M.T. Boedihardjo, and R.~van Handel.
\newblock Matrix concentration inequalities and free probability.
\newblock {\em Inventiones Mathematicae}, 234:419--487, 2023.

\bibitem{BarbMM17}
J.~Barbier, N.~Macris, and L.~Miolane.
\newblock The layered structure of tensor estimation and its mutual
  information.
\newblock In {\em 2017 55th Annual Allerton Conference on Communication,
  Control, and Computing (Allerton)}, pages 1056--1063. IEEE, 2017.

\bibitem{BehneReeves22}
J.~K. Behne and G.~Reeves.
\newblock Fundamental limits for rank-one matrix estimation with groupwise
  heteroskedasticity.
\newblock In {\em Proceedings of The 25th International Conference on
  Artificial Intelligence and Statistics}, volume 151 of {\em Proceedings of
  Machine Learning Research}, pages 8650--8672. PMLR, 2022.

\bibitem{Belinschi17}
S.~T. Belinschi and M.~Capitaine.
\newblock Spectral properties of polynomials in independent {W}igner and
  deterministic matrices.
\newblock {\em Journal of Functional Analysis}, 273(12):3901–3963, December
  2017.

\bibitem{BorCollins19}
C.~Bordenave and B.~Collins.
\newblock Eigenvalues of random lifts and polynomials of random permutation
  matrices.
\newblock {\em Ann. of Math. (2)}, 190(3):811--875, 2019.

\bibitem{BordColl24}
C.~Bordenave and B.~Collins.
\newblock Norm of matrix-valued polynomials in random unitaries and
  permutations.
\newblock 2024.
\newblock available online at \bl{\url{http://arxiv.org/abs/2304.05714}}.

\bibitem{Bourgain99}
J.~Bourgain.
\newblock Random points in isotropic convex sets.
\newblock In {\em Convex Geometric Analysis (Berkeley, CA, 1996)}, volume~34 of
  {\em Mathematical Sciences Research Institute Publications}, pages 53--58.
  Cambridge University Press, 1999.

\bibitem{BrailUniv24}
T.~Brailovskaya and R.~van Handel.
\newblock Universality and sharp matrix concentration inequalities.
\newblock {\em Geometric and Functional Analysis}, 34:1003--1045, 2024.

\bibitem{CaiHanZhang22}
T.~T. Cai, R.~Han, and A.~R. Zhang.
\newblock On the non-asymptotic concentration of heteroskedastic {W}ishart-type
  matrix.
\newblock {\em Electronic Journal of Probability}, 27:1--40, 2022.

\bibitem{Cassidy24}
E.~Cassidy.
\newblock Random permutations acting on $k$-tuples have near-optimal spectral
  gap for $k=\mathrm{poly}(n)$.
\newblock 2024.
\newblock available online at \bl{\url{http://arxiv.org/abs/2412.13941}}.

\bibitem{ChenGVTH26}
C.-F. Chen, J.~Garza-Vargas, J.~A. Tropp, and R.~van Handel.
\newblock A new approach to strong convergence.
\newblock {\em Annals of Mathematics}, 203(2):555--602, 2026.

\bibitem{CollinsGP22}
B.~Collins, A.~Guionnet, and F.~Parraud.
\newblock On the operator norm of non-commutative polynomials in deterministic
  matrices and iid {GUE} matrices.
\newblock {\em Cambridge Journal of Mathematics}, 10(1):195--260, 2022.

\bibitem{CollMale14}
B.~Collins and C.~Male.
\newblock The strong asymptotic freeness of {H}aar and deterministic matrices.
\newblock {\em Annales scientifiques de l'{E}cole Normale Superieure},
  47(1):147--163, 2014.

\bibitem{CollYama26}
B.~Collins and Y.~Yamagishi.
\newblock A {S}udakov-{F}ernique proof of {L}ehner-type edge bounds for
  matrix-valued {GUE} sums.
\newblock 2026.
\newblock available online at \bl{\url{http://arxiv.org/abs/2606.21137}}.

\bibitem{DonGavJohn18}
D.~L. Donoho, M.~Gavish, and I.~M. Johnstone.
\newblock Optimal shrinkage of eigenvalues in the spiked covariance model.
\newblock {\em The Annals of Statistics}, 46(4):1742--1778, 2018.

\bibitem{Ducatez24}
R.~Ducatez, A.~Guionnet, and J.~Husson.
\newblock Large deviation principle for the largest eigenvalue of random
  matrices with a variance profile.
\newblock 2024.
\newblock available online at \bl{\url{http://arxiv.org/abs/2403.05413}}.

\bibitem{Fernique74}
X.~Fernique.
\newblock Des resultats nouveaux sur les processus {G}aussiens.
\newblock {\em C.R. Acad. Sci. Paris Ser A-B}, 278:A363--A365, 1974.

\bibitem{Fernique75}
X.~Fernique.
\newblock Regularite des trajectoires des fonctions aleatoires {G}aussiens.
\newblock {\em Springer Lecture notes}, 480:1--96, 1975.

\bibitem{GianHarTso05}
A.~Giannopoulos, M.~Hartzoulaki, and A.~Tsolomitis.
\newblock Random points in isotropic unconditional convex bodies.
\newblock {\em Journal of the London Mathematical Society}, 72(3):779--798,
  2005.

\bibitem{Gordon85}
Y.~Gordon.
\newblock Some inequalities for {G}aussian processes and applications.
\newblock {\em Israel Journal of Mathematics}, 50(4):265--289, 1985.

\bibitem{GuionnShly09}
A.~Guionnet and D.~Shlyakhtenko.
\newblock Free diffusions and matrix models with strictly convex interaction.
\newblock {\em Geometric and Functional Analysis}, 18(6):1875--1916, 2009.

\bibitem{HaagThor05}
U.~Haagerup and S.~Thorbjornsen.
\newblock A new application of random matrices:
  $\mbox{Ext}(c^*_{\mbox{red}}(f_2))$ is not a group.
\newblock {\em Annals of Mathematics}, 162(2):711--775, 2005.

\bibitem{Hayes22}
B.~Hayes.
\newblock A random matrix approach to the {P}eterson--{T}hom conjecture.
\newblock {\em Indiana Univ. Math. J.}, 71(3):1243--1297, 2022.

\bibitem{HideMagee23}
W.~Hide and M.~Magee.
\newblock Near optimal spectral gaps for hyperbolic surfaces.
\newblock {\em Annals of Mathematics}, 198(2):791--824, 2023.

\bibitem{Husson22}
J.~Husson.
\newblock Large deviations for the largest eigenvalue of matrices with variance
  profiles.
\newblock {\em Electronic Journal of Probability}, 27:1--44, 2022.

\bibitem{HussMcKenna24}
J.~Husson and B.~McKenna.
\newblock Large deviations for the largest eigenvalue of generalized sample
  covariance matrices.
\newblock {\em Electronic Journal of Probability}, 29:1--48, 2024.

\bibitem{KanLovSim97}
R.~Kannan, L.~Lovasz, and M.~Simonovits.
\newblock Random walks and an $o^*(n^5)$ volume algorithm for convex bodies.
\newblock {\em Random Structures \& Algorithms}, 11(1):1--50, 1997.

\bibitem{KolLou17}
V.~Koltchinskii and K.~Lounici.
\newblock Concentration inequalities and moment bounds for sample covariance
  operators.
\newblock {\em Bernoulli}, 23(1):110--133, 2017.

\bibitem{Kunisky26}
D.~Kunisky.
\newblock Lehner’s operator norm formulas, semidefinite programming, and
  spiked matrix models.
\newblock 2026.
\newblock available online at \bl{\url{http://arxiv.org/abs/2606.14687}}.

\bibitem{LatHandYouss18}
R.~Latala, R.~van Handel, and P.~Youssef.
\newblock The dimension-free structure of nonhomogeneous random matrices.
\newblock {\em Inventiones Mathematicae}, 214(2):1031--1080, 2018.

\bibitem{LeeLee24}
S.~Lee and J.~O. Lee.
\newblock Phase transition for the generalized two-community stochastic block
  model.
\newblock {\em J. Appl. Probab.}, 61:385--400, 2024.

\bibitem{Lehner98}
F.~Lehner.
\newblock Free operators with operator coefficients.
\newblock {\em Colloquium Mathematicae}, 74(2):321--328, 1998.

\bibitem{Lehner99}
F.~Lehner.
\newblock Computing norms of free operators with matrix coefficients.
\newblock {\em American Journal of Mathematics}, 121(3):453--486, 1999.

\bibitem{Lehner01}
F.~Lehner.
\newblock On the computation of spectra in free probability.
\newblock {\em J. Funct. Anal.}, 183(2):451--471, 2001.

\bibitem{Lesetal17}
T.~Lesieur, L.~Miolane, M.~Lelarge, F.~Krzakala, and L.~Zdeborová.
\newblock Statistical and computational phase transitions in spiked tensor
  estimation.
\newblock In {\em 2017 IEEE International Symposium on Information Theory
  (ISIT)}, pages 511--515, 2017.

\bibitem{LubPer16}
E.~Lubetzky and Y.~Peres.
\newblock Cutoff on all {R}amanujan graphs.
\newblock {\em Geometric and Functional Analysis}, 26(4):1148--1177, 2016.

\bibitem{Magee25}
M.~Magee.
\newblock Strong convergence of unitary and permutation representations of
  discrete groups.
\newblock 2025.
\newblock available online at \bl{\url{http://arxiv.org/abs/2503.21619}}.

\bibitem{Male12}
C.~Male.
\newblock The norm of polynomials in large random and deterministic matrices.
\newblock {\em Probability Theory and Related Fields}, 154(3--4):659--750,
  2012.

\bibitem{MontRich14}
A.~Montanari and E.~Richard.
\newblock A statistical model for tensor {PCA}.
\newblock {\em Advances in Neural Information Processing Systems}, 27, 2014.

\bibitem{PakKK23}
A.~Pak, J.~Ko, and F.~Krzakala.
\newblock Optimal algorithms for the inhomogeneous spiked {W}igner model.
\newblock In {\em Advances in Neural Information Processing Systems},
  volume~36, pages 5557--5586, 2023.

\bibitem{Paouris06}
G.~Paouris.
\newblock Concentration of mass on convex bodies.
\newblock {\em Geometric and Functional Analysis}, 16(5):1021--1049, 2006.

\bibitem{ParHand25}
E.~Parmaksiz and R.~van Handel.
\newblock Computing extreme singular values of free operators.
\newblock 2025.
\newblock available online at \bl{\url{http://arxiv.org/abs/2510.23987}}.

\bibitem{Parraud22}
F.~Parraud.
\newblock On the operator norm of non-commutative polynomials in deterministic
  matrices and iid {H}aar unitary matrices.
\newblock {\em Probability Theory and Related Fields}, 182(3-4):1019--1073,
  2022.

\bibitem{PerryWB20}
A.~Perry, A.~S. Wein, and A.~S. Bandeira.
\newblock Statistical limits of spiked tensor models.
\newblock {\em Annales de l'Institut Henri Poincare, Probabilites et
  Statistiques}, 56(1):238--295, 2020.

\bibitem{PourBM24}
F.~Pourkamali, J.~Barbier, and N.~Macris.
\newblock Matrix inference in growing rank regimes.
\newblock {\em {IEEE} Trans. Inf. Theory}, 70(11):8133--8163, 2024.

\bibitem{Rudelson99}
M.~Rudelson.
\newblock Random vectors in the isotropic position.
\newblock {\em Journal of Functional Analysis}, 164(1):60--72, 1999.

\bibitem{Schultz05}
H.~Schultz.
\newblock Non-commutative polynomials of independent {G}aussian random
  matrices: {T}he real and symplectic cases.
\newblock {\em Probability Theory and Related Fields}, 131(2):261--309, 2005.

\bibitem{Slep62}
D.~Slepian.
\newblock The one sided barier problem for {G}aussian noise.
\newblock {\em Bell System Tech. Journal}, 41:463--501, 1962.

\bibitem{StojnicCSetam09}
M.~Stojnic.
\newblock Various thresholds for $\ell_1$-optimization in compressed sensing.
\newblock 2009.
\newblock available online at \bl{\url{http://arxiv.org/abs/0907.3666}}.

\bibitem{StojnicRegRndDlt10}
M.~Stojnic.
\newblock Regularly random duality.
\newblock 2013.
\newblock available online at \bl{\url{http://arxiv.org/abs/1303.7295}}.

\bibitem{Stojnicgscompyx16}
M.~Stojnic.
\newblock Fully bilinear generic and lifted random processes comparisons.
\newblock 2016.
\newblock available online at \bl{\url{http://arxiv.org/abs/1612.08516}}.

\bibitem{Stojnicgscomp16}
M.~Stojnic.
\newblock Generic and lifted probabilistic comparisons -- max replaces minmax.
\newblock 2016.
\newblock available online at \bl{\url{http://arxiv.org/abs/1612.08506}}.

\bibitem{Stojnicclupsk25}
M.~Stojnic.
\newblock A {CLuP} algorithm to practically achieve $\sim 0.76$ {SK}--model
  ground state free energy.
\newblock {\em Journal of Statistical Mechanics: Theory and Experiment},
  (11):123302, 2025.

\bibitem{Stojnicalgbp25}
M.~Stojnic.
\newblock Binary perceptron computational gap -- a parametric fl-{RDT} view.
\newblock {\em Journal of Statistical Mechanics: Theory and Experiment},
  (4):043301, 2026.

\bibitem{Stojniccovmat26}
M.~Stojnic.
\newblock Precise sample covariance spectral norm error -- an {RDT} view.
\newblock 2026.
\newblock available online at \bl{\url{http://arxiv.org/abs/2607.14460}}.

\bibitem{Stojnicleh26}
M.~Stojnic.
\newblock An {RDT} based confirmation of {L}ehner's formula for
  {K}ronecker-{G}aussian matrices.
\newblock 2026.
\newblock available online at \bl{\url{http://arxiv.org/abs/2607.26551}}.

\bibitem{Sudakov71}
V.~N. Sudakov.
\newblock Gaussian random processes and measures of solid angles in {H}ilbert
  space.
\newblock {\em Soviet Math. Dokl.}, 12(1):412--415, 1971.

\bibitem{Tropp11}
J.~A. Tropp.
\newblock User-friendly tail bounds for sums of random matrices.
\newblock {\em Foundations of Computational Mathematics}, 11(4):373--434, 2011.

\bibitem{VanHandelSpec17}
R.~van Handel.
\newblock On the spectral norm of {G}aussian random matrices.
\newblock {\em Transactions of the American Mathematical Society},
  369(11):8161--8178, 2017.

\bibitem{HandelSurvey26}
R.~van Handel.
\newblock Strong convergence: A short survey.
\newblock In {\em Proceedings of the International Congress of Mathematicians},
  pages 145--165. SIAM, 2026.

\bibitem{VershNonAsym12}
R.~Vershynin.
\newblock Introduction to the non-asymptotic analysis of random matrices.
\newblock In {\em Compressed Sensing}, pages 210--268. Cambridge University
  Press, 2012.

\bibitem{Vitale00}
R.~A. Vitale.
\newblock Some comparisons for {G}aussian processes.
\newblock {\em Proceedings of the American Mathematical Society},
  128(10):3043--3046, 2000.

\bibitem{Voic91}
D.~Voiculescu.
\newblock Limit laws for random matrices and free products.
\newblock {\em Invent. Math.}, 104(1):201--220, 1991.

\end{thebibliography}
\end{singlespace}

\end{document}